\documentclass[11pt, oneside]{article}
\usepackage{geometry}                	
\usepackage{graphicx}	
\usepackage{amssymb}
\usepackage{mleftright}
\usepackage{amsmath,amsthm,mathtools,color,mathabx}
\usepackage{usebib}
\usepackage{esint}
\newtheorem{theorem}{Theorem}[section]
\newtheorem{corollary}[theorem]{Corollary}
\newtheorem{lemma}[theorem]{Lemma}
\newtheorem{proposition}[theorem]{Proposition}
\newtheorem{remark}{Remark}
\newtheorem{definition}[theorem]{Definition}

\newcommand{\divx}{\mathop{\mathrm{div}}}
\newcommand{\esssup}{\mathop{\mathrm{ess~sup}}}

\allowdisplaybreaks[4]

\numberwithin{equation}{section}

\def\Yint#1{\mathchoice
    {\YYint\displaystyle\textstyle{#1}}%
    {\YYint\textstyle\scriptstyle{#1}}%
    {\YYint\scriptstyle\scriptscriptstyle{#1}}%
    {\YYint\scriptscriptstyle\scriptscriptstyle{#1}}%
      \!\iint}
\def\YYint#1#2#3{{\setbox0=\hbox{$#1{#2#3}{\iint}$}
    \vcenter{\hbox{$#2#3$}}\kern-.51\wd0}}
\def\longdash{{-}\mkern-3.5mu{-}} 
\def\tiltlongdash{\rotatebox[origin=c]{15}{$\longdash$}}
\def\fiint{\Yint\tiltlongdash}

\title{Gradient continuity for the parabolic $(1,\,p)$-Laplace equation under the subcritical case}
\author{Shuntaro Tsubouchi \footnote{Graduate School of Mathematical Sciences, The University of Tokyo, Japan. \textit{Email}: \texttt{tsubos@g.ecc.u-tokyo.ac.jp}}}
\date{}
\begin{document}
\maketitle

\abstract{This paper is concerned with the gradient continuity for the parabolic $(1,\,p)$-Laplace equation. In the supercritical case $\frac{2n}{n+2}<p<\infty$, where $n\ge 2$ denotes the space dimension, this gradient regularity result has been proved recently by the author. In this paper, we would like to prove that the same regularity holds even for the subcritical case $1<p\le \frac{2n}{n+2}$ with $n\ge 3$, on the condition that a weak solution admits the $L^{s}$-integrability with $s>n(2-p)/p$. The gradient continuity is proved, similarly to the supercritical case, once the local gradient bounds of solutions are verified. Hence, this paper mainly aims to show the local boundedness of a solution and its gradient by Moser's iteration. The proof is completed by considering a parabolic approximate problem, verifying a comparison principle, and showing a priori gradient estimates of a bounded weak solution to the relaxed equation.}
\bigbreak

\textbf{Mathematics Subject Classification (2020)} 35K92, 35B65, 35A35
\bigbreak

\textbf{Keywords} Gradient continuity, Local gradient bounds, Moser's iteration, Comparison principle

\section{Introduction}\label{Sect:Introduction}
This paper is concerned with the parabolic regularity for a weak solution to the $(1,\,p)$-Laplace equation
\begin{equation}\label{Eq (Section 1): (1,p)-Laplace}
\partial_{t}u-\Delta_{1}u-\Delta_{p}u=0\quad \textrm{in}\quad \Omega_{T}\coloneqq \Omega\times (0,\,T),
\end{equation}
where $\Omega\subset {\mathbb R}^{n}$ is a bounded Lipschitz domain, and $T\in(0,\,\infty)$ is a fixed constant.
For an unknown function $u=u(x_{1},\,\dots\,,\,x_{n},\,t)$, the time derivative and the spatial gradient of $u$ are respectively denoted by $\partial_{t}u$ and $\nabla u=(\partial_{x_{j}}u)_{j=1,\,\dots\,,\,n}$.
The divergence operators $\Delta_{1}$ and $\Delta_{p}$ are the one-Laplacian and the $p$-Laplacian, defined as 
\[\Delta_{s}u\coloneqq \divx\left(\lvert\nabla u\rvert^{s-2}\nabla u \right)\quad \textrm{for}\quad s\in\lbrack 1,\,\infty).\]
In this paper, the space dimension $n$ and the exponent $p$ are assumed to be
\begin{equation}\label{Eq (Section 1): Subcritical Range}
n\ge 3\quad \textrm{and}\quad 1<p\le \frac{2n}{n+2}.
\end{equation}
We aim to prove that $\nabla u$ is continuous in $\Omega_{T}$, provided the unknown function $u\colon \Omega_{T}\to{\mathbb R}$ satisfies
\begin{equation}\label{Eq (Section 1): Higher Integrability}
u\in L_{\mathrm{loc}}^{s}(\Omega_{T})\quad \textrm{with}\quad s>s_{\mathrm{c}}\coloneqq\frac{n(2-p)}{p}.
\end{equation}
The higher integrability assumption (\ref{Eq (Section 1): Higher Integrability}) is optimal, since otherwise any improved regularity is in general not expected for the parabolic $p$-Laplace equation for $p\in(1,\, \frac{2n}{n+2}\rbrack$.

In \cite{T-parabolic}, the author has recently shown the same regularity result for
\begin{equation}\label{Eq (Section 1): Parabolic (1,p)-Poisson}
  \partial_{t}u-\Delta_{1}u-\Delta_{p}u=f\quad \text{in}\quad \Omega_{T},
\end{equation}
where $n$ and $p$ satisfy the supercritical case
\begin{equation}\label{Eq (Section 1): Supercritical Range}
  n\ge 2\quad \text{and}\quad \frac{2n}{n+2}<p<\infty,
\end{equation}
and the external force term $f\in L^{r}(\Omega_{T})$ is given with the exponent $r$ suitably large.
Compared to \cite{T-parabolic}, this paper, which deals with the subcritical case (\ref{Eq (Section 1): Subcritical Range}), requires $f\equiv 0$ for a technical issue (see Section \ref{Subsect: Lit}).
\subsection{Truncation approach}\label{Subsect: St}
In Section \ref{Subsect: St}, we mention the basic strategy for showing $\nabla u\in C^{0}(\Omega_{T};\,{\mathbb R}^{n})$ in \cite{T-parabolic}.
More detailed explanations are given in Section \ref{Subsect: A priori Hoelder}.

The main difficulty arises from the fact that the uniform parabolicity of $-\Delta_{1}-\Delta_{p}$ breaks as a gradient vanishes.
To explain this, we formally differentiate (\ref{Eq (Section 1): (1,p)-Laplace}) by the space variable $x_{j}$.
The resulting equation is 
\begin{equation}\label{Eq (Section 1): Differentiated Eq}
\partial_{t}\partial_{x_{j}}u-\mathop{\mathrm{div}}(\nabla^{2} E(\nabla u)\nabla\partial_{x_{j}}u)=0,
\end{equation}
where $E(z)=\lvert z\rvert+\lvert z\rvert^{p}/p\,(z\in{\mathbb R}^{n})$ is the energy density.
The coefficient matrix $\nabla^{2}E(\nabla u)$ loses its uniform ellipticity on the facet $\{\nabla u=0\}$, in the sense that the ratio
\begin{equation}\label{Eq (Section 1): PR}
  \frac{\text{(the largest eigenvalue of $\nabla^{2} E(\nabla u)$)}}{\text{(the smallest eigenvalue of $\nabla^{2} E(\nabla u)$)}}=C_{p}\left(1+\lvert \nabla u\rvert^{1-p}\right)
\end{equation}
blows up as $\nabla u\to 0$.
In this sense, (\ref{Eq (Section 1): (1,p)-Laplace}) is not everywhere uniformly parabolic, which makes it difficult to deduce quantitative continuity estimates for $\nabla u$, especially on the facet.
However, the ratio above will be bounded, if a gradient does not vanish.
Hence, we introduce a truncated spatial gradient
\[{\mathcal G}_{\delta}(\nabla u)\coloneqq (\lvert\nabla u\rvert-\delta)_{+}\frac{\nabla u}{\lvert \nabla u\rvert},\]
where $\delta\in(0,\,1)$ denotes the truncation parameter, and $a_{+}\coloneqq \max\{\,a,\,0\,\}\equiv a\wedge 0$ for $a\in{\mathbb R}$.
The main purpose in \cite{T-parabolic} is to prove the local H\"{o}lder continuity of ${\mathcal G}_{\delta}(\nabla u)$, whose continuity estimate may depend on $\delta$.
This is naturally expected since the support of ${\mathcal G}_{\delta}(\nabla u)$ coincides with a non-degenerate regime $\{\lvert\nabla u\rvert\ge \delta>0\}$, where we can regard (\ref{Eq (Section 1): (1,p)-Laplace}) uniformly parabolic, depending on each $\delta$.
Although our continuity estimate of ${\mathcal G}_{\delta}(\nabla u)$ will depend on $\delta$, it is possible to conclude the continuity of ${\mathcal G}_{0}(\nabla u)=\nabla u$, since the mapping ${\mathcal G}_{\delta}$ uniformly approximates the identity.
In this qualitative way, we complete the proof of the gradient continuity.
This truncation approach can be found in the recent study of elliptic regularity for the second-order $(1,\,p)$-Laplace problem (\cite{T-scalar, T-system}) and a second-order degenerate problem (\cite{BDGPdN, CF MR3133426}; see also \cite{SV MR2728558} for a weaker result).

To achieve our goal rigorously, we have to appeal to approximate (\ref{Eq (Section 1): (1,p)-Laplace}).
Here we should note that (\ref{Eq (Section 1): (1,p)-Laplace}) is not uniformly parabolic, especially on the facet.
This prevents us from applying a standard difference quotient method, and hence it seems difficult to treat (\ref{Eq (Section 1): Differentiated Eq}) in $L^{2}(0,\,T;\,W^{-1,\,2}(\Omega))$.
For this reason, we have to consider a parabolic approximate equation that is uniformly parabolic, depending on the approximation parameter $\varepsilon\in(0,\,1)$.
In this paper, we relax the energy density $E(z)=\lvert z\rvert+\lvert z\rvert^{p}/p$ by convoluting with the Friedrichs mollifier $\rho_{\varepsilon}$ (see \cite{MR2916967} as a related item).
Therefore, we consider an approximate equation of the form
\[\partial_{t}u_{\varepsilon}-\divx(\nabla E^{\varepsilon}  (\nabla u_{\varepsilon}))=0\quad \text{with}\quad E^{\varepsilon}  \coloneqq \rho_{\varepsilon}\ast E\in C^{\infty}({\mathbb R}^{n}).\]
The proof is completed by showing the $L^{p}$-strong convergence of a gradient, and the local H\"{o}lder continuity of
\[{\mathcal G}_{2\delta,\,\varepsilon}(\nabla u_{\varepsilon})\coloneqq \left(\sqrt{\varepsilon^{2}+\lvert\nabla u_{\varepsilon}\rvert^{2}}-2\delta\right)_{+}\frac{\nabla u_{\varepsilon}}{\lvert\nabla u_{\varepsilon}\rvert},\]
whose continuity estimate may depend on $\delta\in(0,\,1)$ but is independent of $\varepsilon\in(0,\,\delta/8)$.

The detailed computations of the H\"{o}lder gradient estimates of ${\mathcal G}_{2\delta,\,\varepsilon}(\nabla u_{\varepsilon})$ are already given in \cite[Theorem 2.8]{T-parabolic} for $p\in(1,\,\infty)$, provided that $\nabla u_{\varepsilon}$ is uniformly bounded with respect to $\varepsilon\in(0,\,\delta/8)$ (see Section \ref{Sect:MainTheorem}).
Therefore, it suffices to prove local gradient bounds of $u_{\varepsilon}$, which is the main purpose of this paper and is proved by following the three steps.
Firstly, we show local $L^{\infty}$-bounds of $u$ by Moser's iteration, where (\ref{Eq (Section 1): Higher Integrability}) is used (see also \cite[Theorem 2]{Choe MR1135917} and \cite[Chapter 8, A.2]{MR2865434}).
Secondly, we verify a comparison principle and a weak maximum principle for $u_{\varepsilon}$ under some Dirichlet boundary conditions.
For this topic, we refer the reader to \cite[Chapter 4]{BDGLS} and \cite[Chapter 3]{MR2356201}, which materials provide comparison principles for weak solutions.
Finally, we prove $\nabla u_{\varepsilon}\in L_{\mathrm{loc}}^{q}$ for any $q\in(p,\,\infty\rbrack$ by $u_{\varepsilon}\in L_{\mathrm{loc}}^{\infty}$ and Moser's iteration.
The recent item \cite[Chapter 9]{BDGLS} gives a similar result of gradient bounds for parabolic $p$-Laplace equations with $p\in(1,\,2)$.
The main difference between this paper and \cite[Chapter 9]{BDGLS} is that we carefully choose test functions that are always supported in a certain non-degenerate region of $u_{\varepsilon}$.

\subsection{Literature overview}\label{Subsect: Lit}
We briefly mention some literature on the parabolic $p$-Laplace equation, and the source of (\ref{Eq (Section 1): (1,p)-Laplace}).
Also, we would like to compare this paper with the author's recent paper \cite{T-parabolic}.

For the parabolic $p$-Laplace equation
\begin{equation}\label{Eq (Section 1): parabolic p}
  \partial_{t}u-\Delta_{p}u=0,
\end{equation} 
the existence and the regularity of a weak solution $u$ are well-established.
The existence theory is found in the monographs \cite{Lions MR0259693, Showalter MR1422252}, based on the Faedo--Galerkin method and the monotonicity of $\Delta_{p}$.
There, (\ref{Eq (Section 1): parabolic p}) is treated in $L^{p^{\prime}}(0,\,T;\,W^{-1,\,p^{\prime}}(\Omega))$ when $\frac{2n}{n+2}<p<\infty$, and in $L^{p^{\prime}}(0,\,T;\,W^{-1,\,p^{\prime}}(\Omega)+L^{2}(\Omega))$ when $1<p\le \frac{2n}{n+2}$, where $p^{\prime}\coloneqq p/(p-1)$ denotes the H\"{o}lder conjugate exponent of $p$.
The H\"{o}lder gradient continuity of $u$ was proved by DiBenedetto--Friedman \cite{MR783531, MR814022} in 1985 for the supercritical range $p\in(\frac{2n}{n+2},\,\infty)$ (see also \cite{MR0718944, MR743967, MR0886719} for weaker results).
Later in 1991, Choe \cite{Choe MR1135917} proved the same regularity result for $p\in(1,\,\infty)$, under the assumption that $u$ is in $L_{\mathrm{loc}}^{s}(\Omega_{T})$ with the exponent $s\in(1,\,\infty)$ satisfying $n(p-2)+sp>0$.
In particular, \cite{Choe MR1135917} covers the subcritical case (\ref{Eq (Section 1): Subcritical Range}) with the higher integrability assumption (\ref{Eq (Section 1): Higher Integrability}).
It is worth mentioning that without (\ref{Eq (Section 1): Higher Integrability}), no improved regularity result is expected even for the $p$-Laplace problem (see \cite{Dibenedetto Herrero MR1066761}).
In these fundamental works, careful scaling arguments in space and time are used, so that the H\"{o}lder gradient continuity estimates are quantitatively deduced.
This is often called the intrinsic scaling method, which plays an important role when showing various regularity properties for (\ref{Eq (Section 1): parabolic p}).
As related materials, see the monographs \cite{MR1230384, MR2865434}, and the recent paper \cite{BDLS}.

The sources of $(1,\,p)$-Laplace problem can be found in the fields of fluid mechanics for $p=2$ \cite[Chapter VI]{MR0521262}, and materials sciences for $p=3$ \cite{spohn1993surface}.
Among them, the second-order parabolic equation (\ref{Eq (Section 1): Parabolic (1,p)-Poisson}) can be found when modeling the motion of the Bingham fluid, the non-Newtonian fluid that has both plastic and viscosity properties.
In this model, the one-Laplacian $\Delta_{1}$ reflects the plasticity of a fluid, while the Laplacian $\Delta=\Delta_{2}$ does the viscosity.
As explained in \cite[\S 1.3]{T-parabolic} (see also \cite[Chapter VI]{MR0521262}), (\ref{Eq (Section 1): Parabolic (1,p)-Poisson}) arises when one considers the unknown three-dimensional vector field $U=(0,\,0,\,u(t,\,x_{1},\,x_{2}))$ denoting the velocity of a Bingham fluid in a pipe cylinder $\Omega\times {\mathbb R}\subset {\mathbb R}^{2}\times {\mathbb R}$.
There, the external force term $f=-\partial_{x_{3}}\pi$, where $\pi$ denotes the pressure function, depends at most on $t$.
Mathematical analysis for the $(1,\,p)$-Laplace equation (\ref{Eq (Section 1): (1,p)-Laplace}) at least goes back to \cite{MR0521262}, where the methods based on variational inequalities are used.
However, the continuity of a spatial gradient for (\ref{Eq (Section 1): (1,p)-Laplace}) has not been well-established, even for $p=2$.

Motivated by the Bingham fluid model, in \cite{T-parabolic}, the author has recently shown the gradient continuity for (\ref{Eq (Section 1): Parabolic (1,p)-Poisson}).
More precisely, \cite{T-parabolic} treats the case where the conditions (\ref{Eq (Section 1): Supercritical Range}) and $f\in L^{r}(\Omega_{T})$ are satisfied with 
\[\frac{1}{p}+\frac{1}{r}\le 1\quad \text{and}\quad n+2<r\le\infty.\]
By the former assumption, the continuous inclusion $L^{q}(\Omega_{T})\hookrightarrow L^{p^{\prime}}(0,\,T;\,W^{-1,\,p^{\prime}}(\Omega))$ holds. 
This inclusion plays an important role in constructing the solution $u$, and it appears that the former condition cannot be removed when showing convergence for (\ref{Eq (Section 1): Approximation general}).
It is worth noting that the latter assumption is optimal when one considers the gradient continuity for parabolic $p$-Laplace equations with external force terms \cite[Chapters VIII--IX]{MR1230384}.
For the approximation of (\ref{Eq (Section 1): Parabolic (1,p)-Poisson}), the following equation is considered in \cite[\S 2]{T-parabolic};
\begin{equation}\label{Eq (Section 1): Approximation general}
  \partial_{t}u_{\varepsilon}-\divx(\nabla E^{\varepsilon}  (\nabla u_{\varepsilon}))=f_{\varepsilon}\quad \text{in}\quad \Omega_{T},
\end{equation}
where $f_{\varepsilon}\in C^{\infty}(\Omega_{T})$ weakly converges to $f$ in $L^{r}(\Omega_{T})$.
In the supercritical case (\ref{Eq (Section 1): Supercritical Range}), the compact embedding $V_{0}\coloneqq W_{0}^{1,\,p}(\Omega)\hookrightarrow\hookrightarrow L^{2}(\Omega)$ and the continuous inclusion $L^{2}(\Omega)\hookrightarrow W^{-1,\,p^{\prime}}(\Omega)\eqqcolon V_{0}^{\prime}$ hold.
In particular, we are allowed to use the parabolic compact embedding
\begin{equation}\label{Eq (Section 1): CPT}
    \left\{ u\in L^{p}(0,\,T;\,V_{0})\mathrel{}\middle|\mathrel{} \partial_{t} u\in L^{p^{\prime}}(0,\,T;\,V_{0}^{\prime}) \right\}\hookrightarrow\hookrightarrow L^{p}(0,\,T;\,L^{2}(\Omega))
\end{equation}
by the Aubin--Lions lemma \cite[Chapter III, Proposition 1.3]{Showalter MR1422252}.
The strong convergence result for weak solutions to (\ref{Eq (Section 1): Approximation general}) is shown in \cite[Proposition 2.4]{T-parabolic}, where the compact embedding (\ref{Eq (Section 1): CPT}) is carefully used in controlling a non-zero external force term $f_{\varepsilon}$.
In the subcritical case (\ref{Eq (Section 1): Subcritical Range}), however, we have to choose $V_{0}\coloneqq W_{0}^{1,\,p}(\Omega)\cap L^{2}(\Omega)$ in constructing a weak solution (see \cite{Lions MR0259693}).
Since the continuous inclusion $V_{0}\hookrightarrow L^{2}(\Omega)$ is no longer compact, we cannot use (\ref{Eq (Section 1): CPT}).
For this technical reason, we have to let $f_{\varepsilon}\equiv 0$ when (\ref{Eq (Section 1): Subcritical Range}) is assumed.

The basic strategy of this paper is almost the same as that of \cite{T-parabolic}.
However, there are mainly two differences between this paper and \cite{T-parabolic}, which respectively deal with the cases (\ref{Eq (Section 1): Subcritical Range}) and (\ref{Eq (Section 1): Supercritical Range}).
The first is that, as explained in the last paragraph, it appears that our approximation argument might not work in treating any non-trivial external force term in the case (\ref{Eq (Section 1): Subcritical Range}).
The second is that the local $L^{\infty}$-bound of $\nabla u_{\varepsilon}$ basically depends on that of $u_{\varepsilon}$ when $p\in(1,\,\frac{2n}{n+2}\rbrack$, while for $p\in (\frac{2n}{n+2},\,\infty)$, local $L^{p}$--$L^{\infty}$ estimates of $\nabla u_{\varepsilon}$ are straightforwardly shown in \cite[Theorem 2.7]{T-parabolic}.
It should be emphasized that there is no difference between (\ref{Eq (Section 1): Subcritical Range}) and (\ref{Eq (Section 1): Supercritical Range}), when it comes to showing the a priori continuity estimates that depend on the truncation parameter $\delta$.
Indeed, for (\ref{Eq (Section 1): Approximation general}) with $p\in(1,\,\infty)$ and $r\in(n+2,\,\infty\rbrack$, the uniform local H\"{o}lder estimates of ${\mathcal G}_{2\delta,\,\varepsilon}(\nabla u_{\varepsilon})$ are already given in \cite[Theorem 2.8]{T-parabolic}, provided that $\nabla u_{\varepsilon}$ admits uniform $L^{\infty}$-bounds.
Therefore, in this paper, we mainly focus on the strong convergence and the local bound of $\nabla u_{\varepsilon}$.

\subsection{Notations}
Before stating our main result, we would like to fix some notations.

The set of all natural numbers is denoted by ${\mathbb N}\coloneqq \{\,1,\,2,\,\dots\,\}$, and we write ${\mathbb Z}_{\ge 0}\coloneqq \{0\}\cup {\mathbb Z}$, the set of all non-negative integers.
The symbol ${\mathbb R}$ denotes the set of all real numbers, and the $n$-dimensional Euclidean space is denoted by ${\mathbb R}^{n}$.
We often use the symbols ${\mathbb R}_{>0}\coloneqq (0,\,\infty)\subset{\mathbb R}$, and ${\mathbb R}_{\ge 0}\coloneqq \lbrack 0,\,\infty)\subset{\mathbb R}$.
For a fixed exponent $s\in(1,\,\infty)$, the symbol $s^{\prime}\coloneqq s/(s-1)\in(1,\,\infty)$ stands for its H\"{o}lder conjugate exponent.

For the vectors $x=(x_{1},\,\dots\,,\,x_{n}),\,y=(y_{1},\,\dots\,,\,y_{n})\in{\mathbb R}^{n}$, the standard inner product and the Euclidean norm are defined as
\[\langle x\mid y\rangle\coloneqq \sum_{j=1}^{n}x_{j}y_{j}\in{\mathbb R}\quad \text{and}\quad \lvert x\rvert\coloneqq \sqrt{\langle x\mid x\rangle}\in{\mathbb R}_{\ge 0}.\]

For $R\in (0,\,\infty)$, $x_{0}\in{\mathbb R}^{n}$, $t_{0}\in{\mathbb R}$, we write the open ball $B_{R}(x_{0})\coloneqq \{x\in{\mathbb R}^{n}\mid \lvert x-x_{0}\rvert<R\}$, and the half interval $I_{R}(t_{0})\coloneqq (t_{0}-R^{2},\,t_{0}\rbrack$.
The symbol $Q_{R}(x_{0},\,t_{0})\coloneqq B_{R}(x_{0})\times I_{R}(t_{0}) \subset{\mathbb R}^{n+1}$ stands for the parabolic cylinder centered at $(x_{0},\,t_{0})$.
When the points $x_{0}\in{\mathbb R}^{n}$ and $t_{0}\in{\mathbb R}$ are clear, the symbols $B_{R}$, $I_{R}$, and $Q_{R}$ are used for notational simplicity. 
For $U\subset{\mathbb R}^{k}$ with $k\in{\mathbb N}$, the boundary of $U$ and the closure of $U$, with respect to the Euclidean metric, are denoted by $\partial U$ and $\overline{U}$.
For $\Omega_{T}=\Omega\times (0,\,T)$ and $Q_{R}(x_{0},\,t_{0})=B_{R}(x_{0})\times I_{R}(t_{0})$, the parabolic boundaries are straightforwardly defined as $\partial_{\mathrm p}\Omega_{T}\coloneqq (\Omega\times \{0\})\cup (\partial\Omega\times\lbrack 0,\,T))\subset{\mathbb R}^{n+1}$ and $\partial_{\mathrm p}Q_{R}\coloneqq (B_{R}(x_{0})\times \{t_{0}-R^{2}\})\cup ( \partial B_{R}(x_{0})\times \lbrack t_{0}-R^{2},\,t_{0}))\subset{\mathbb R}^{n+1}$.
We introduce the standard parabolic metric $d_{\mathrm p}$ as 
\[d_{\mathrm p}((x,\,t),\,(y,\,s))\coloneqq \max\left\{\lvert x-y\rvert,\,\lvert t-s\rvert^{1/2}\right\}\]
for $(x,\,t),\,(y,\,s)\in \Omega_{T}$. With respect to this metric, we can define the parabolic distance $\mathop{\mathrm{dist}}_{\mathrm p}(A,\,B)$ for $A,\,B\subset \Omega_{T}$.
For $A,\,B\subset \Omega_{T}$, we say $A\Subset B$ when $\overline{A}\subset B$ holds.

For an $n\times n$ real matrix $A=(A_{j,\,k})_{j,\,k}$, we use the standard norms
\[\lVert A\rVert\coloneqq \sup\left\{ \lvert Ax\rvert\mathrel{}\middle|\mathrel{}x\in{\mathbb R}^{n},\,\lvert x\rvert\le 1 \right\},\quad \text{and} \quad \lvert A\rvert\coloneqq \sqrt{\sum_{j,\,k=1}^{n}A_{j,\,k}^{2}},\]
often called the operator norm and the Hilbert--Schmidt norm respectively.
For $n\times n$ real symmetric matrices $A,\,B$, we say $A\leqslant B$ when $B-A$ is positive semi-definite.
The identity matrix, and the zero matrix are respectively denoted by $\mathrm{id}_{n}$, and $O_{n}$.

For a real Banach space $E$ and its continuous dual space $E^{\prime}$, the duality pairng is denoted by $\langle f,\,v\rangle_{E^{\prime},\,E}\coloneqq f(v)\in{\mathbb R}$ for $f\in E^{\prime}$ and $v\in E$.
When the spaces $E$ and $E^{\prime}$ are clear, we write $\langle f,\,v\rangle$ for simplicity of notation.

Fix a Lebesugue measurable set $U\subset{\mathbb R}^{l}$ with $l\in{\mathbb N}$.
The symbol $L^{s}(U;\,{\mathbb R}^{k})$ denotes the Lebesgue space for $k\in{\mathbb N}$ and $s\in\lbrack 1,\,\infty\rbrack$.
In other words, an ${\mathbb R}^{k}$-valued function $f=f(x)$ is in $L^{s}(U;\,{\mathbb R}^{k})$ if and only if it is Lebesgue measurable in $U$, and the norm 
\[\lVert f\rVert_{L^{s}(U)}\coloneqq \left\{\begin{array}{rl}\left( \displaystyle\int_{U}\lvert f\rvert^{s}\,{\mathrm d}x\right)^{1/s}& (1\le s<\infty),\\ \esssup\limits_{x\in U}\lvert f(x)\rvert & (s=\infty), \end{array} \right.\] 
is finite.
For a real-valued function $f=f(x,\,t)$ that is integrable in a parabolic cylinder $Q_{r}=Q_{r}(x_{0},\,t_{0})\subset{\mathbb R}^{n+1}$, we set an average integral
\[\fiint_{Q_{r}}f(x,\,t)\,{\mathrm d}x{\mathrm d}t\coloneqq \frac{1}{\lvert Q_{r}\rvert}\iint_{Q_{r}}f(x,\,t){\mathrm d}x{\mathrm d}t\in{\mathbb R}.\]
For $s\in(1,\,\infty)$, $I=(a,\,b)\subset{\mathbb R}$, and a real Banach space $E$, the symbol $L^{s}(I;\,E)=L^{s}(a,\,b;\,E)$ denotes the standard Bochnar space, equipped with the norm
\[\lVert u\rVert_{L^{s}(I;\,E)}\coloneqq \left(\int_{a}^{b}\lVert u(t)\rVert_{E}^{s}\,{\mathrm d}t\right)^{1/s}.\]

For given $s\in \lbrack 1,\,\infty\rbrack$, $k\in{\mathbb Z}_{\ge 0}$, and a bounded domain $U\Subset {\mathbb R}^{n}$, let $W^{k,\,s}(U)$ denote the Sobolev space of the $k$-th order.
In other words, a real-valued function $f$ is in $W^{k,\,s}(U)$ if and only if any $l$-th order partial derivative of $f$ with $l\in\{\,0,\,\dots\,k\,\}$, in the sense of distribution, is in $L^{s}(U)$. We write $W^{k,\,s}(U;\,{\mathbb R}^{l})\coloneqq \left(W^{k,\,s}(U)\right)^{l}$ for $l\in{\mathbb N}$.
The norm of function space $W^{1,\,s}(U)$ is equipped with the norm $\lVert f\rVert_{W^{1,\,s}(U)}\coloneqq\lVert f\rVert_{L^{s}(U)}+\lVert \nabla f\rVert_{L^{s}(U)}$ for each $f\in W^{1,\,s}(U)$.
The function space $W_{0}^{1,\,s}(U)$ is defined as the closure of $C_{\mathrm c}^{1}(U)\subset W^{1,\,s}(U)$, the set of all real-valued functions that are both compactly supported in $U$, and continuously differentiable in $U$.
This function space is equipped with another norm $\lVert \nabla f\rVert_{L^{s}(U)}$ for $f\in W_{0}^{1,\,s}(U)$, which is equivalent to $\lVert\,\cdot\, \rVert_{W^{1,\,s}(U)}$.
The continuous dual space of $(W_{0}^{1,\,s}(U),\,\lVert \nabla \,\cdot\,\rVert_{L^{s}(U)})$ is denoted by $W^{-1,\,s^{\prime}}(U)$.

For $l,\,m\in {\mathbb N}$, and an open set $U\subset{\mathbb R}^{k}$ with $k\in{\mathbb N}$, let $C^{m}(U;\,{\mathbb R}^{l})$ denote the set of all ${\mathbb R}^{l}$-valued function of the $C^{m}$-class.
Also, for any set $U\subset{\mathbb R}^{k}$, we write $C^{0}(U;\,{\mathbb R}^{l})$ for the set of all ${\mathbb R}^{l}$-valued continuous mapping in $U$.
For $l=1$, these function spaces are denoted by $C^{m}(U)$ and $C^{0}(U)$ for short.
For a closed interval $I\subset{\mathbb R}$, the symbol $C^{0}(I;\,L^{2}(\Omega))$ stands for the set of all $L^{2}(\Omega)$-valued functions in $I$ that are strongly continuous.

\subsection{Main result and outline of the paper}
In this paper, we consider a generalized equation of the form
\begin{equation}\label{Eq (Section 1): main eq}
\partial_{t}u-\divx\left(\nabla E(\nabla u)\right)=0
\end{equation}
in $\Omega_{T}$ with $E=E_{1}+E_{p}$, where $E_{1}$ and $E_{p}$ are convex mappings from ${\mathbb R}^{n}$ to ${\mathbb R}_{\ge 0}$.
For the smoothness of these densities, we require $E_{1}\in C^{0}({\mathbb R}^{n})\cap C^{2}({\mathbb R}^{n}\setminus \{0\})$ and $E_{p}\in C^{1}({\mathbb R}^{n})\cap C^{2}({\mathbb R}^{n}\setminus \{ 0\})$.
The density $E_{p}$ admits the constants $0<\lambda_{0}\le\Lambda_{0}<\infty$ satisfying 
\begin{equation}\label{Eq (Section 1): C-1-bound of Ep}
\lvert \nabla E_{p}(z)\rvert\le \Lambda_{0}\lvert z\rvert^{p-1}
\end{equation}
for all $z\in{\mathbb R}^{n}$, and
\begin{equation}\label{Eq (Section 1): Ellipticity of Ep}
\lambda_{0}\lvert z\rvert^{p-2}\mathrm{id}_{n} \leqslant \nabla^{2} E_{p}(z)\leqslant \Lambda_{0}\lvert z\rvert^{p-2}\mathrm{id}_{n}
\end{equation}
for all $z\in{\mathbb R}^{n}\setminus \{ 0\}$.
We assume that $E_{1}$ is positively one-homogeneous.
More precisely, $E_{1}$ satisfies
\begin{equation}\label{Eq (Section 1): Positive 1-hom}
E_{1}(k z)=kE_{1}(z)
\end{equation}
for all $z\in{\mathbb R}^{n}$, $k\in{\mathbb R}_{>0}$.
For the continuity of the Hessian matrices of $E_{p}$, we assume that there exists a concave, non-decreasing function $\omega_{p}\colon {\mathbb R}_{\ge 0}\to {\mathbb R}_{\ge 0}$ with $\omega_{p}(0)=0$, such that 
 \begin{equation}\label{Eq (Section 1): modulus of continuity of Hess Ep}
  \left\lVert \nabla^{2}E_{p}(z_{1})-\nabla^{2}E_{p}(z_{2}) \right\rVert\le C_{\delta,\,M}\omega_{p}(\lvert z_{1}-z_{2}\rvert/\mu)
 \end{equation}
 holds for all $z_{1}$, $z_{2}\in{\mathbb R}^{n}$ with $\mu/32\le \lvert z_{j}\rvert\le 3\mu$ for $j\in\{\,1,\,2\,\}$, and $\mu\in(\delta,\,M-\delta)$.
 Here $\delta$ and $M$ are fixed constants such that $0<2\delta<M<\infty$, and the constant $C_{\delta,\,M}\in{\mathbb R}_{>0}$ depends on $\delta$ and $M$.
 For $E_{1}$, we require the existence of a concave, non-decreasing function $\omega_{1}\colon{\mathbb R}_{\ge 0}\to {\mathbb R}_{\ge 0}$ with $\omega_{1}(0)=0$, such that
 \begin{equation}\label{Eq (Section 1): (Assumption) modulus of continuity of Hess E1}
  \left\lVert \nabla^{2}E_{1}(z_{1})-\nabla^{2}E_{1}(z_{2}) \right\rVert\le \omega_{1}(\lvert z_{1}-z_{2}\rvert)
 \end{equation}
 holds for all $z_{1}$, $z_{2}\in{\mathbb R}^{n}$ with $1/32\le \lvert z_{j}\rvert\le 3$ for $j\in\{\,1,\,2\,\}$.
 Although the assumptions (\ref{Eq (Section 1): modulus of continuity of Hess Ep})--(\ref{Eq (Section 1): (Assumption) modulus of continuity of Hess E1}) are used not in showing local gradient bounds, they are needed in the proof of a priori H\"{o}lder estimates of truncated gradients.
 Since this paper mainly aims to show local gradient bounds, (\ref{Eq (Section 1): modulus of continuity of Hess Ep})--(\ref{Eq (Section 1): (Assumption) modulus of continuity of Hess E1}) are not explicitly used, except last Section \ref{Sect:MainTheorem}.

To define a weak solution to (\ref{Eq (Section 1): main eq}), we introduce standard function spaces.
For $p\in (1,\,\frac{2n}{n+2}\rbrack$, we set 
\[V_{0}\coloneqq W_{0}^{1,\,p}(\Omega)\cap L^{2}(\Omega)\]
equipped with the norm
\[\lVert v\rVert_{V_{0}}\coloneqq \lVert \nabla v \rVert_{L^{p}(\Omega)}+\lVert v\rVert_{L^{2}(\Omega)}\]
for $v\in V_{0}$.
Then, the continuous dual space of $V_{0}$ is \(V_{0}^{\prime}=W^{-1,\,p^{\prime}}(\Omega)+L^{2}(\Omega).\)

We set the parabolic function spaces
\[\begin{array}{rcl} X^{p}(0,\,T;\,\Omega)&\coloneqq & \left\{ u\in L^{p}(0,\,T;\,V)\mathrel{}\middle|\mathrel{} \partial_{t} u\in L^{p^{\prime}}(0,\,T;\,V_{0}^{\prime}) \right\},\\ X_{0}^{p}(0,\,T;\,\Omega)&\coloneqq & \left\{ u\in L^{p}(0,\,T;\,V_{0})\mathrel{}\middle|\mathrel{} \partial_{t} u\in L^{p^{\prime}}(0,\,T;\,V_{0}^{\prime}) \right\}, \end{array}\]
where $V\coloneqq W^{1,\,p}(\Omega)\cap L^{2}(\Omega)$.
From the Gelfand triple $V_{0}\hookrightarrow L^{2}(\Omega)\hookrightarrow V_{0}^{\prime}$, the inclusion $X_{0}^{p}(0,\,T;\,\Omega)\subset C^{0}(\lbrack 0,\,T\rbrack;\,L^{2}(\Omega))$ follows by the Lions--Magenes lemma \cite[Chapter III, Proposition 1.2]{Showalter MR1422252}.
\begin{definition}\label{Def: Weak sol}\upshape
A function $u\in X^{p}(0,\,T;\,\Omega)\cap C^{0}(\lbrack 0,\,T\rbrack;\,L^{2}(\Omega))$ is called a weak solution to (\ref{Eq (Section 1): main eq}) when there exists $Z\in L^{\infty}(\Omega_{T};\,{\mathbb R}^{n})$ such that
\begin{equation}\label{Eq (Section 1): Subgradient}
Z(x,\,t)\in \partial E_{1}(\nabla u(x,\,t))\quad \text{for a.e.~}(x,\,t)\in\Omega_{T},
\end{equation}
and
\begin{equation}\label{Eq (Section 1): Weak form in def}
\int_{0}^{t}\langle \partial_{t}u,\,\varphi\rangle_{V_{0}^{\prime},\,V_{0}}\,{\mathrm d}t+\iint_{\Omega_{T}}\left\langle Z+\nabla E_{p}(\nabla u)\mathrel{}\middle|\mathrel{} \nabla\varphi\right\rangle\,{\mathrm d}x {\mathrm d}t=0
\end{equation}
for all $\varphi\in X_{0}^{p}(0,\,T;\,\Omega)$. Here $\partial E_{1}$ denotes the subdifferential of $E_{1}$, defined as
\[\partial E_{1}(z)\coloneqq \left\{\zeta\in {\mathbb R}^{n}\mathrel{}\middle|\mathrel{} E_{1}(w)\ge E_{1}(z)+\langle \zeta\mid w-z\rangle\text{ for all }w\in{\mathbb R}^{n}\right\}\quad \text{for }z\in{\mathbb R}^{n}.\]
\end{definition}

The main result is the following Theorem \ref{Thm}.
\begin{theorem}\label{Thm}
Let $n$, $p$ and $E=E_{1}+E_{p}$ satisfy (\ref{Eq (Section 1): Subcritical Range}) and (\ref{Eq (Section 1): C-1-bound of Ep})--(\ref{Eq (Section 1): (Assumption) modulus of continuity of Hess E1}).
Assume that a function $u$ is a weak solution to (\ref{Eq (Section 1): main eq}) in $\Omega_{T}$. If (\ref{Eq (Section 1): Higher Integrability}) is satisfied, then the spatial gradient $\nabla u$ is continuous in $\Omega_{T}$.
\end{theorem}

The contents of this paper are as follows.
In Section \ref{Sect:Pre}, we briefly mention basic properties of $E_{1}$ and some composite functions.
There, we also note fundamental iteration lemmata, which are fully used in a priori bound estimates.
Section \ref{Sect:Approximation} mainly provides the strong convergence of a parabolic approximate equation under a suitable Dirichlet boundary condition (Proposition \ref{Prop: Convergence}).
There, some basic properties concerning $E^{\varepsilon}  $ are also mentioned.
Section \ref{Sect:Bound of u} aims to verify the local $L^{\infty}$-bound of $u$ and $u_{\varepsilon}$.
The former is shown by Moser's iteration in Section \ref{Subsect: Moser}.
The latter is proved by the comparison principle (Proposition \ref{Prop: CP}) and the weak maximum principle (Corollary \ref{Lemma: WMP}) in Section \ref{Subsect: CP}.
Section \ref{Sect: L-p to L-infty} establishes local $L^{q}$-bounds of $\nabla u_{\varepsilon}$ for $q\in(p,\,\infty\rbrack$.
After deducing local energy estimates in Section \ref{Subsect: Energy}, we complete the case $q\in(p,\,\infty)$ by the condition $u_{\varepsilon}\in L_{\mathrm{loc}}^{\infty}$ (Proposition \ref{Prop: Higher integrability}), and the remaining one $q=\infty$ by Moser's iteration (Proposition \ref{Prop: Moser's iteration}) in Section \ref{Sect:GradientBound}.
The main result in Section \ref{Sect: L-p to L-infty} is that the uniform local bound of $\nabla u_{\varepsilon}$ follow from that of $u_{\varepsilon}$ and the uniform $L^{p}$-bound of $\nabla u_{\varepsilon}$ (Theorem \ref{Thm: Grad Bound}).
Section \ref{Sect:MainTheorem} aims to show Theorem \ref{Thm}.
There, a priori H\"{o}lder estimates of ${\mathcal G}_{2\delta,\,\varepsilon}(\nabla u_{\varepsilon})$ (Theorem \ref{Thm: A priori Hoelder}) is used without proof, since this is already shown in \cite[Theorem 2.8]{T-parabolic}.
In Section \ref{Subsect: A priori Hoelder}, however, we would like to mention brief sketches of the proof of Theorem \ref{Thm: A priori Hoelder} for the reader's convenience.
Finally in Section \ref{Subsect: Last Subsection}, we give the proof of Theorem \ref{Thm} by Proposition \ref{Prop: Convergence}, Corollary \ref{Lemma: WMP}, Theorems \ref{Thm: Grad Bound} and \ref{Thm: A priori Hoelder}.

\section{Preliminary}\label{Sect:Pre}
\subsection{Basic properties of positively one-homogeneous density}
We briefly note the basic properties of the positively one-homogeneous function $E_{1}$.

The subdifferential $\partial E_{1}$ is given by
\[\partial E_{1}(z)=\left\{\zeta\in{\mathbb R}^{n}\mathrel{}\middle|\mathrel{} \langle\zeta \mid z \rangle=E_{1}(z),\, \langle w\mid z\rangle\le 1\text{ for all }w\in C_{E_{1}}\right\},\]
where $C_{E_{1}}\coloneqq \{w\in{\mathbb R}^{n}\mid E_{1}(w)\le 1\}$ (see \cite[Theorem 1.8]{MR2033382}).
In particular, for any vector fields $\nabla u\in L^{1}(\Omega_{T};\,{\mathbb R}^{n})$ and $Z\in L^{\infty}(\Omega_{T};\,{\mathbb R}^{n})$ that satisfy (\ref{Eq (Section 1): Subgradient}), there holds
\begin{equation}\label{Eq (Section 2): Euler's id}
  \langle Z\mid \nabla u\rangle=E_{1}(\nabla u)\quad \text{a.e.~in }\Omega_{T},
\end{equation}
which is often called Euler's identity.

By (\ref{Eq (Section 1): Positive 1-hom}), it is easy to check that
\[\nabla E_{1}(kz)=\nabla E_{1}(z)\quad \text{and}\quad \nabla^{2}E_{1}(kz)=k^{-1}\nabla^{2}E_{1}(z)\]
for all $z\in{\mathbb R}^{n}\setminus\{0\}$ and $k\in(0,\,\infty)$.
In particular, we have
\begin{equation}\label{Eq (Section 2): Bound of Z}
  \partial E_{1}(z)\subset \left\{w\in {\mathbb R}^{n}\mathrel{}\middle|\mathrel{}\lvert w\rvert\le K_{0} \right\}\quad \text{for all }z\in{\mathbb R}^{n},
\end{equation}
\begin{equation}\label{Eq (Section 2): Bound of nabla E1}
  \lvert \nabla E_{1}(z)\rvert\le K_{0}\quad \text{for all }z\in{\mathbb R}^{n}\setminus\{ 0\},
\end{equation}
\begin{equation}\label{Eq (Section 2): Order of E1 Hessian}
  O_{n}\leqslant \nabla^{2}E_{1}(z)\leqslant \frac{K_{0}}{\lvert z\rvert}\mathrm{id}_{n}\quad \text{for all }z\in{\mathbb R}^{n}\setminus\{0\},
\end{equation}
for some constant $K_{0}\in(0,\,\infty)$.

\subsection{Composite functions}\label{Subsect: psi}
Throughout this paper, we let $\psi\colon{\mathbb R}_{\ge 0}\to {\mathbb R}_{\ge 0}$ be a bounded function that is continuously differentiable in ${\mathbb R}_{>0}$ except at finitely many points.
Also, we assume that the derivative $\psi^{\prime}$ is non-negative and its support is compactly supported in ${\mathbb R}_{\ge 0}$.
In particular, $\psi$ is non-decreasing, and becomes constant for sufficiently large $\sigma$.
Corresponding to this $\psi$, we define the convex function $\Psi\colon{\mathbb R}_{\ge 0}\to{\mathbb R}_{\ge 0}$ as
\begin{equation}\label{Eq (Section 2): Convex Psi}
\Psi(\sigma)\coloneqq \int_{0}^{\sigma}\tau\psi(\tau)\,{\mathrm d}\tau\quad \text{for }\sigma\in{\mathbb R}_{\ge 0}.
\end{equation}
By the definition of $\Psi$ and the monotonicity of $\psi$, it is clear that
\begin{equation}\label{Eq (Section 2): Psi-bounds}
\Psi(\sigma)\le \sigma^{2}\psi(\sigma)\quad \text{for all }\sigma\in{\mathbb R}_{\ge 0}.
\end{equation}

In our proof of local bound estimates, we mainly choose $\psi\colon{\mathbb R}_{\ge 0}\to {\mathbb R}_{\ge 0}$ as either
\begin{equation}\label{Eq (Section 2): psi-alpha-M}
\psi_{\alpha,\,M}(\sigma)\coloneqq (\sigma\wedge M)^{\alpha}\quad \text{for }\sigma\in{\mathbb R}_{\ge 0}
\end{equation}
or 
\begin{equation}\label{Eq (Section 2): psi-alpha-M tilde}
{\tilde \psi}_{\alpha,\,M}(\sigma)\coloneqq \sigma^{\alpha}\left(1-\sigma^{-1}\right)_{+}\wedge M^{\alpha}\left(1-M^{-1}\right)\quad \text{for }\sigma\in{\mathbb R}_{\ge 0},
\end{equation}
where $\alpha\in\lbrack 0,\,\infty)$ and $M\in(1,\,\infty)$.
For $\psi_{\alpha,\,M}$ or ${\tilde \psi}_{\alpha,\,M}$, the correponding $\Psi$ defined as (\ref{Eq (Section 2): Convex Psi}) is denoted by $\Psi_{\alpha,\,M}$ or ${\tilde \Psi}_{\alpha,\,M}$ respectively.
When $M\to\infty$, the monotone convergences
\begin{equation}\label{Eq (Section 2): Monotone conv of psi-s}
\psi_{\alpha,\,M}(\sigma)\nearrow \sigma^{\alpha},\quad {\tilde\psi}_{\alpha,\,M}(\sigma)\nearrow \sigma^{\alpha}\left(1-\sigma^{-1}\right)_{+}, \quad \Psi_{\alpha,\,M}(\sigma)\nearrow \frac{\sigma^{\alpha+2}}{\alpha+2}
\end{equation}
hold for every $\sigma\in{\mathbb R}_{\ge 0}$, which result is used later when showing various local $L^{\infty}$-estimates.
Also, by direct computations, we can easily notice the following (\ref{Eq (Section 2): psi estimate 0})--(\ref{Eq (Section 2): psi-tilde estimate 2}):
\begin{equation}\label{Eq (Section 2): psi estimate 0}
  \left(\psi_{\alpha,\,M}(\sigma)\right)^{r}\left(\psi_{\alpha,\,M}^{\prime}(\sigma)\right)^{1-r}\sigma^{r} \le \alpha^{r}\psi_{\alpha,\,M}(\sigma)\quad \textrm{for all }\sigma\in {\mathbb R}_{>0} \setminus\{\,1,\,M\,\},
\end{equation}
\begin{equation}\label{Eq (Section 2): psi-tilde estimate 1}
{\tilde\psi}_{\alpha,\,M}(\sigma)+\sigma{\tilde \psi}_{\alpha,\,M}^{\prime}(\sigma)\le (\alpha+1)\psi_{\alpha,\,M}(\sigma)\chi_{\{\sigma>1\}}(\sigma)\quad \textrm{for all }\sigma\in {\mathbb R}_{>0} \setminus\{\,1,\,M\,\},
\end{equation}
\begin{equation}\label{Eq (Section 2): psi-tilde estimate 2}
  \sigma^{\alpha+2}\le 1+\sigma^{\alpha+p}+\lim_{M\to\infty}\left(\sigma^{2}{\tilde \psi}_{\alpha,\,M}(\sigma)\right)\quad \text{for all }\sigma\in{\mathbb R}_{\ge 0}.
\end{equation}
Here $r\in(1,\,\infty)$ is a fixed constant.

\subsection{Iteration lemmata}
Without proofs, we infer two basic lemmata, shown by standard iteration arguments (see \cite[Lemma V.3.1]{MR717034} and \cite[Lemma 4.2]{T-parabolic} for the proof).
\begin{lemma}\label{Lemma: Absorbing Iteration}
Fix $R_{1},\,R_{2}\in{\mathbb R}_{>0}$ with $R_{1}<R_{2}$.
Assume that a bounded function $f\colon \lbrack R_{1},\,R_{2} \rbrack\to{\mathbb R}_{\ge 0}$ admits the constants $A,\,\alpha\in {\mathbb R}_{>0}$, $B\in{\mathbb R}_{\ge 0}$, and $\theta\in (0,\,1)$, such that there holds
\[f(r_{1})\le \theta f(r_{2})+\frac{A}{(r_{2}-r_{1})^{\alpha}}+B\]
for any $r_{1},\,r_{2}\in\lbrack R_{1},\,R_{2}\rbrack$ with $r_{1}<r_{2}$.
Then, $f$ satisfies
\[f(R_{1})\le C(\alpha,\,\theta)\left[\frac{A}{(R_{2}-R_{1})^{\alpha}}+B \right].\]
\end{lemma}
\begin{lemma}\label{Lemma: Moser iteration lemma}
Let $\kappa\in(1,\,\infty)$ be a constant.
Assume that the sequences $\{Y_{l}\}_{l=0}^{\infty}\subset {\mathbb R}_{\ge 0}$, $\{p_{l}\}_{l=0}^{\infty}\subset \lbrack 1,\,\infty)$ admit the constants $A$, $B\in(1,\,\infty)$ and $\mu\in{\mathbb R}_{>0}$ such that
\[\left\{\begin{array}{rcl} Y_{l+1}^{p_{l+1}}& \le & \left(AB^{l}Y_{l}^{p_{l}} \right)^{\kappa}, \\ p_{l}& \ge & \mu\left(\kappa^{l}-1\right), \end{array}  \right. \quad \text{for all }l\in{\mathbb Z}_{\ge 0},\]
and $\kappa^{l}p_{l}^{-1}\to \mu^{-1}$ as $l\to\infty$ .  
Then, we have
\[\limsup_{l\to\infty}Y_{l}\le A^{\frac{\kappa^{\prime}}{\mu}}B^{\frac{(\kappa^{\prime})^{2}}{\mu}}Y_{0}^{\frac{p_{0}}{\mu}},\]
where $\kappa^{\prime}\coloneqq \kappa/(\kappa-1)\in(1,\,\infty)$ denotes the H\"{o}lder conjugate exponent of $\kappa$.
\end{lemma}

\section{Approximation problem}\label{Sect:Approximation}
\subsection{Approximation of energy density}
We would like to explain the approximation of $E=E_{1}+E_{p}$, based on the Friedrichs mollifier.
More precisely, we introduce a non-negative, spherically symmetric function $\rho\in C_{\mathrm c}^{\infty}({\mathbb R}^{n})$ such that $\lVert \rho\rVert_{L^{1}}=1$ and the support of $\rho$ is the closed unit ball centered at the origin. 
For the approximation parameter $\varepsilon\in(0,\,1)$, we define \(\rho_{\varepsilon}(z)\coloneqq \varepsilon^{-n}\rho(z/\varepsilon)\) for $z\in{\mathbb R}^{n}$, and relax the energy density $E_{s}\, (s\in\{\,1,\,p\,\})$ by the non-negative function, defined as
\begin{equation}\label{Eq (Section 3): Def of E-s-epsilon}
  E_{s,\,\varepsilon}(z)\coloneqq \int_{{\mathbb R}^{n}}\rho_{\varepsilon}(y)E_{s}(z-y)\,{\mathrm d}y\quad \textrm{for }z\in{\mathbb R}^{n}.
\end{equation}
Then, by (\ref{Eq (Section 1): C-1-bound of Ep})--(\ref{Eq (Section 1): Ellipticity of Ep}) and (\ref{Eq (Section 2): Bound of nabla E1})--(\ref{Eq (Section 2): Order of E1 Hessian}), the relaxed density $E^{\varepsilon}  \coloneqq E_{1,\,\varepsilon}+E_{p,\,\varepsilon}$ satisfy
\begin{equation}\label{Eq (Section 3): Bounds of nabla E-epsilon}
  \lvert \nabla E^{\varepsilon}  (z)\rvert\le \Lambda(\varepsilon+\lvert z\rvert^{2})^{(p-1)/2}+K,
\end{equation}
\begin{equation}\label{Eq (Section 3): Estimates on eigenvalues}
  \lambda \left(\varepsilon^{2}+\lvert z\rvert^{2} \right)^{p/2-1}\mathrm{id}_{n}\leqslant \nabla^{2}E^{\varepsilon}  (z)\leqslant \left(\Lambda \left(\varepsilon^{2}+\lvert z\rvert^{2} \right)^{p/2-1}+\frac{K}{\sqrt{\varepsilon^{2}+\lvert z\rvert^{2}}}\right)\mathrm{id}_{n}
\end{equation}
\begin{equation}\label{Eq (Section 3): Strong Monotonicity of E-p-epsilon}
  \left\langle \nabla E_{p,\,\varepsilon}(z)-\nabla E_{p,\,\varepsilon}(w)\mathrel{}\middle|\mathrel{}z-w \right\rangle \ge \lambda\left(\varepsilon^{2}+\lvert z\rvert^{2}+\lvert w\rvert^{2}\right)^{p/2-1}\lvert z-w\rvert^{2},
\end{equation}
\begin{equation}\label{Eq (Section 3): Strong Monotonicity of E-epsilon}
  \left\langle \nabla E^{\varepsilon}  (z)-\nabla E^{\varepsilon}  (w)\mathrel{}\middle|\mathrel{}z-w \right\rangle \ge \lambda\left(\varepsilon^{2}+\lvert z\rvert^{2}+\lvert w\rvert^{2}\right)^{p/2-1}\lvert z-w\rvert^{2},
\end{equation}
for all $z$, $w\in{\mathbb R}^{n}$.
Here $\lambda\in(0,\,\lambda_{0})$, $\Lambda\in(\Lambda_{0},\,\infty)$, and $K\in(K_{0},\,\infty)$ are constants (see \cite[\S 2]{T-scalar} for the detailed computations).
Letting $w=0$ and $\varepsilon\to 0$ in (\ref{Eq (Section 3): Strong Monotonicity of E-p-epsilon}), we have 
\begin{equation}\label{Eq (Section 3): Strong Monotonicity of E-p}
  \left\langle \nabla E_{p}(z)\mathrel{} \middle|\mathrel{}z\right\rangle\ge \lambda\lvert z\rvert^{p}\quad \text{for all }z\in{\mathbb R}^{n},
\end{equation}
which follows from $E_{p}\in C^{1}({\mathbb R}^{n})$ and $\nabla E_{p}(0)=0$.
Also, letting $w=0$ in (\ref{Eq (Section 3): Strong Monotonicity of E-epsilon}), we get 
\begin{equation}\label{Eq (Section 3): Strong Monotonicity of E-epsilon modified}
\left\langle \nabla E^{\varepsilon}  (z)-\nabla E^{\varepsilon}  (0)\mathrel{}\middle|\mathrel{}z \right\rangle\ge \lambda \left(\lvert z\rvert^{p}-\varepsilon^{p}\right)\quad \text{for all }z\in{\mathbb R}^{n}.
\end{equation}

As a special case of \cite[Lemma 2.8]{T-scalar}, we can use the following lemma.
 \begin{lemma}\label{Lemma: Fundamental Conv}
 The energy density $E^{\varepsilon}  =E_{1,\,\varepsilon}+E_{p,\,\varepsilon}\in C^{\infty}({\mathbb R}^{n})$, defined as (\ref{Eq (Section 3): Def of E-s-epsilon}) for each $\varepsilon\in(0,\,1)$, satisfies the following.
  \begin{enumerate}
    \item\label{Item 1/2} For each fixed $v\in L^{p}(\Omega_{T};\,{\mathbb R}^{n})$, we have 
    \[\nabla E^{\varepsilon}  (v)\to A_{0}(v)\quad \textrm{in}\quad L^{p^{\prime}}(\Omega_{T};\,{\mathbb R}^{n})\quad \textrm{as}\quad \varepsilon\to 0.\]
    Here the mapping $A_{0}\colon{\mathbb R}^{n}\rightarrow {\mathbb R}^{n}$ is defined as
    \[A_{0}\coloneqq \left\{\begin{array}{cc}
      \nabla E(z) & (z\neq 0), \\ 
      (\rho\ast \nabla E_{1})(0) & (z=0).
    \end{array} \right.\]
    \item\label{Item 2/2} Assume that a sequence $\{v_{\varepsilon_{k}}\}_{k}\subset L^{p}(\Omega_{T};\,{\mathbb R}^{n})$, where $\varepsilon_{k}\to 0$ as $k\to 0$, satisfies
      \[v_{\varepsilon_{k}}\to v_{0}\quad \textrm{in}\quad L^{p}(\Omega_{T};\,{\mathbb R}^{n})\quad \textrm{as}\quad k\to\infty\]
      for some $v_{0}\in L^{p}(\Omega_{T})$.
      Then, up to a subsequence, we have
      \[\left\{\begin{array}{rclcc}
        \nabla E_{p,\,\varepsilon_{k}}(v_{\varepsilon_{k}})&\to& \nabla E_{p}(v_{0})& \textrm{in}& L^{p^{\prime}}(\Omega_{T};\,{\mathbb R}^{n}),\\ 
      \nabla E_{1,\,\varepsilon_{k}}(v_{\varepsilon_{k}}) &\stackrel{}{\rightharpoonup} & Z & \textrm{in} &L^{\infty} (\Omega_{T};\,{\mathbb R}^{n}), 
      \end{array}  \right.\quad \textrm{as}\quad k\to\infty.\]
      Here the limit $Z\in L^{\infty}(\Omega_{T};\,{\mathbb R}^{n})$ satisfies 
      \[Z(x,\,t)\in\partial E_{1}(v_{0}(x,\,t))\quad \textrm{for a.e.~}(x,\,t)\in\Omega_{T}.\]
     \end{enumerate}
 \end{lemma}

\subsection{Convergence of approximate solutions}
Section \ref{Sect:Approximation} is concluded by verifying that $u_{\varepsilon}$, a weak solution to
\begin{equation}\label{Eq (Section 3): Approximate Equation}
\partial_{t}u_{\varepsilon}-\divx\left(\nabla E^{\varepsilon}  (\nabla u_{\varepsilon}) \right)=0
\end{equation}
in $\Omega_{T}$, converges to a weak solution to (\ref{Eq (Section 1): main eq}).

Let $u_{\star}\in X^{p}(0,\,T;\,\Omega)\cap C(\lbrack 0,\,T\rbrack;\,L^{2}(\Omega))$ be a fixed function.
By carrying out similar arguments in \cite{Lions MR0259693, Showalter MR1422252}, we find the unique weak solution of
\begin{equation}\label{Eq (Section 3): Dirichlet Prob approx}
\left\{\begin{array}{rclcc} \partial_{t}u_{\varepsilon}-\divx(\nabla E^{\varepsilon}  (\nabla u_{\varepsilon}))&=&0 & \text{in} &\Omega_{T},\\ 
u_{\varepsilon} &=&u_{\star} & \text{on} & \partial_{\mathrm{p}}\Omega_{T}.  \end{array} \right.
\end{equation}
More precisely, the solution $u_{\varepsilon}$ is in $u_{\star}+X_{0}^{p}(0,\,T;\,\Omega)$, satisfies (\ref{Eq (Section 3): Approximate Equation}) in $L^{p^{\prime}}(0,\,T;\,V_{0}^{\prime})$, and $(u_{\varepsilon}-u_{\star})(\,\cdot\,,\,0)=0$ in $L^{2}(\Omega)$.
The existence of the weak solution of (\ref{Eq (Section 3): Dirichlet Prob approx}) is shown by the Faedo--Galerkin method (see \cite[Chapitre 2]{Lions MR0259693}, \cite[\S III.4]{Showalter MR1422252} as related materials).
We would like to prove that $u_{\varepsilon}$ converges to the weak solution of
\begin{equation}\label{Eq (Section 3): Dirichlet Prob}
\left\{\begin{array}{rclcc} \partial_{t}u-\divx(\nabla E(\nabla u))&=&0 & \text{in} &\Omega_{T},\\ 
u &=&u_{\star} & \text{on} & \partial_{\mathrm{p}}\Omega_{T},  \end{array} \right.
\end{equation}
in the sense of Definition \ref{Def: Dirichlet sol} below.

\begin{definition}\label{Def: Dirichlet sol}\upshape
  Let $u_{\star}\in X^{p}(0,\,T;\,\Omega)\cap C^{0}(\lbrack 0,\,T\rbrack;\,L^{2}(\Omega))$.
  A function $u\in X^{p}(0,\,T;\,\Omega)$ is called the weak solution of (\ref{Eq (Section 3): Dirichlet Prob}) when the following two properties are satisfied.
  \begin{enumerate}
    \item $u\in u_{\star}+X_{0}^{p}(0,\,T;\,\Omega)\subset C^{0}(\lbrack 0,\,T\rbrack;\,L^{2}(\Omega))$ and $(u_{\varepsilon}-u_{\star})(\,\cdot\,,\,0)=0$ in $L^{2}(\Omega)$.
    \item $u_{\varepsilon}$ is a weak solution to (\ref{Eq (Section 1): main eq}) in $\Omega_{T}$ in the sense of Definition \ref{Def: Weak sol}.
  \end{enumerate}
\end{definition}
By a weak compactness argument and Lemma \ref{Lemma: Fundamental Conv}, we prove Proposition \ref{Prop: Convergence}.
\begin{proposition}\label{Prop: Convergence}
Fix arbitrary $u_{\star}\in X^{p}(0,\,T;\,\Omega)\cap C^{0}(\lbrack 0,\,T\rbrack;\,L^{2}(\Omega))$ and $\tau\in(0,\,T)$.
Let $u_{\varepsilon}\in u_{\star}+X_{0}^{p}(0,\,T;\,\Omega)$ be the unique weak solution of (\ref{Eq (Section 3): Dirichlet Prob approx}) for each $\varepsilon\in(0,\,1)$.
Then, there exists a decreasing sequence $\{\varepsilon_{k}\}_{k}\subset (0,\,1)$ such that 
\[\nabla u_{\varepsilon_{k}}\to \nabla u_{0}\quad \textrm{a.e.~in}\quad \Omega_{T-\tau}\quad \textrm{and}\quad  \textrm{strongly in}\quad L^{p}(\Omega_{T-\tau};\,{\mathbb R}^{n}),\]
where the limit function $u_{0}\in u_{\star}+X_{0}^{p}(0,\,T;\,\Omega)$ is the unique weak solution of (\ref{Eq (Section 3): Dirichlet Prob}) with $T$ replaced by $T-\tau$.
\end{proposition}
Compared to Proposition \ref{Prop: CP}, \cite[Proposition 2.4]{T-parabolic} provides a similar convegence result for (\ref{Eq (Section 1): Approximation general}) in the case (\ref{Eq (Section 1): Supercritical Range}).
There the compact embedding (\ref{Eq (Section 1): CPT}) is used to deal with non-trivial external force terms, as explained in Section \ref{Subsect: Lit}.
Here we give the proof of Proposition \ref{Prop: Convergence} without using any compact embedding.
\begin{proof}
We set \(T_{0}\coloneqq T-\tau/2\), and $T_{1}\coloneqq T-\tau$.
We define 
\[\theta(t)\coloneqq \frac{(t-T_{0})_{+}}{T_{0}}\in\lbrack 0,\,1\rbrack\quad \text{for}\quad t\in\lbrack 0,\,T\rbrack.\]
To construct $u_{0}$, we first claim that
\[{\mathbf J}_{\varepsilon}\coloneqq \lVert u_{\varepsilon}-u_{\star} \rVert_{L^{p}(0,\,T_{1};\,V_{0})}+\lVert \partial_{t}u_{\varepsilon} \rVert_{L^{p^{\prime}}(0,\,T_{1};\,V_{0}^{\prime})}\] 
is bounded, uniformly for $\varepsilon\in (0,\,1)$.
To prove the boundedness of ${\mathbf J}_{\varepsilon}$, we test $\varphi\coloneqq (u_{\varepsilon}-u_{\star}) \theta$ into (\ref{Eq (Section 3): Approximate Equation}).
Integrating by parts, we have
\begin{align*}
& {\mathbf L}_{1}+{\mathbf L}_{2}\\
&\coloneqq -\frac{1}{2}\iint_{\Omega_{T_{0}}}(u_{\varepsilon}-u_{\star})^{2}\partial_{t}\theta\,{\mathrm d}x {\mathrm d}t+\iint_{\Omega_{T_{0}}}\left\langle \nabla E^{\varepsilon}  (\nabla u_{\varepsilon})-\nabla E^{\varepsilon}  (0) \mathrel{}\middle|\mathrel{}\nabla u_{\varepsilon} \right\rangle\theta \,{\mathrm d}x {\mathrm d}t\\
&=-\int_{0}^{T_{0}}\langle \partial_{t}u_{\star},\, (u_{\varepsilon}-u_{\star})\theta\rangle_{V_{0}^{\prime},\,V_{0}}\,{\mathrm d}t+\iint_{\Omega_{T_{0}}}\left\langle \nabla E^{\varepsilon}  (\nabla u_{\varepsilon})-\nabla E^{\varepsilon}  (0) \mathrel{}\middle|\mathrel{}\nabla u_{\star}\right\rangle\theta\,{\mathrm d}x{\mathrm d}t\\ 
&\eqqcolon -{\mathbf R}_{1}+{\mathbf R}_{2}.
\end{align*}
By (\ref{Eq (Section 3): Strong Monotonicity of E-epsilon modified}) and our choice of $\theta$, we have
\[{\mathbf L}_{1}+{\mathbf L}_{2}\ge \frac{1}{2T_{0}}\iint_{\Omega_{T_{0}}}\lvert u_{\varepsilon}-u_{\star}\rvert^{2}\,{\mathrm d}x{\mathrm d}t+\lambda\iint_{\Omega_{T_{0}}}\lvert\nabla u_{\varepsilon}\rvert^{p}\theta\,{\mathrm d}x{\mathrm d}t-\lambda\varepsilon^{p}\lvert \Omega_{T_{0}}\rvert.\]
To estimate ${\mathbf R}_{1}$, we use H\"{o}lder's inequality and Young's inequality to compute
\begin{align*}
\lvert{\mathbf R}_{1}\rvert
&\le \lVert \partial_{t}u_{\star} \rVert_{L^{p^{\prime}}(0,\,T_{0};\,V_{0}^{\prime})}\left(\int_{0}^{T}\left(\lVert \nabla (u_{\varepsilon}-u_{\star})\theta \rVert_{L^{p}(\Omega)}+\lVert (u_{\varepsilon}-u_{\star})\theta \rVert_{L^{2}(\Omega)} \right)^{p}\,{\mathrm d}t\right)^{1/p} \\ 
&\le \frac{\lambda}{3}\iint_{\Omega_{T_{0}}}\lvert\nabla u_{\varepsilon}\rvert^{p}\theta\,{\mathrm d}x{\mathrm d}t+\frac{C_{p}}{\lambda}\lVert\partial_{t}u_{\star}\rVert_{L^{p^{\prime}}(0,\,T_{0};\,V_{0}^{\prime})}^{p^{\prime}}\\ 
&\quad +C_{p}\lVert \partial_{t}u_{\star}\rVert_{L^{p^{\prime}}(0,\,T_{0};\,V_{0}^{\prime})}\lVert\nabla u_{\star}\rVert_{L^{p}(\Omega_{T_{0}})} \\ 
&\quad\quad  +\frac{1}{4T_{0}}\iint_{\Omega_{T_{0}}}\lvert u_{\varepsilon}-u_{\star} \rvert^{2}\,{\mathrm d}x{\mathrm d}t+C_{p}T_{0}^{2/p}\lVert \partial_{t}u_{\star} \rVert_{L^{p^{\prime}}(0,\,T_{0};\,V_{0}^{\prime})}^{2}.
\end{align*}
By (\ref{Eq (Section 3): Bounds of nabla E-epsilon}) and Young's inequality, we have
\begin{align*}
  \lvert{\mathbf R}_{2} \rvert 
  &\le C(\Lambda,\,K)\iint_{\Omega_{T_{0}}}\left(1+\lvert \nabla u_{\varepsilon}\rvert^{p-1} \right)\lvert \nabla u_{\star}\rvert\theta\,{\mathrm d}x{\mathrm d}t\\ 
  &\le \frac{\lambda}{3}\iint_{\Omega_{T_{0}}}\lvert \nabla u_{\varepsilon}\rvert^{p}\theta\,{\mathrm d}x{\mathrm d}t+\frac{C(\Lambda,\,K)}{\lambda}\iint_{\Omega_{T_{0}}}\left(1+\lvert \nabla u_{\star}\rvert^{p}\right)\,{\mathrm d}x{\mathrm d}t.
\end{align*}
Combining these three estimates, we obtain
\[\iint_{\Omega_{T_{0}}}\lvert u_{\varepsilon}-u_{\star}\rvert^{2}\,{\mathrm d}x{\mathrm d}t+\iint_{\Omega_{T_{0}}}\lvert \nabla u_{\varepsilon}\rvert^{p}\theta\,{\mathrm d}x{\mathrm d}t\le {\hat C},\]
where ${\hat C}$ depends on $\lambda$, $\Lambda$, $K$, $T_{0}$, $\lvert \Omega\rvert$, $\lVert \nabla u_{\star} \rVert_{L^{p}(\Omega_{T_{0}})}$, and $\lVert \partial_{t} u_{\star}\rVert_{L^{p^{\prime}}(0,\,T_{0};\,V_{0}^{\prime})}$.
Recalling our choice of $\theta$, we have
\begin{align*}
  &\lVert u_{\varepsilon}-u_{\star} \rVert_{L^{p}(0,\,T_{1};\,V_{0})}^{p}\\ 
  &=\iint_{\Omega_{T_{1}}}\left(\lVert \nabla (u_{\varepsilon}-u_{\star}) \rVert_{L^{p}(\Omega)}+\lVert u_{\varepsilon}-u_{\star} \rVert_{L^{2}(\Omega)} \right)^{p}\,{\mathrm d}t\\ 
  &\le C_{p}\left(\iint_{\Omega_{T_{1}}}\lvert \nabla (u_{\varepsilon}-u_{\star}) \rvert^{p}{\mathrm d}x{\mathrm d}t+T_{1}^{1-p/2}\left(\iint_{\Omega_{T_{1}}}\lvert u_{\varepsilon}-u_{\star}\rvert^{2}\,{\mathrm d}x{\mathrm d}t \right)^{p/2} \right)\\ 
  &\le {\tilde C}(p,\,\tau,\,T,\,{\hat C})<\infty.
\end{align*}
Since $u_{\varepsilon}$ satisfy (\ref{Eq (Section 3): Approximate Equation}) in $L^{p^{\prime}}(0,\,T;\,V_{0}^{\prime})$, for any $\varphi\in L^{p}(0,\,T;\,V_{0})$, we have
\begin{align*}
\left\lvert \int_{0}^{T_{1}}\langle \partial_{t}u_{\varepsilon},\,\varphi \rangle_{V_{0}^{\prime},\,V_{0}}\,{\mathrm d}t\right\rvert&
\le \lVert \nabla E^{\varepsilon}  (\nabla u_{\varepsilon})\rVert_{L^{p^{\prime}}(\Omega_{T_{1}})}\lVert \nabla \varphi\rVert_{L^{p}(\Omega_{T_{1}})}\\ 
&\le {\check C}(p,\,\Lambda,\,K,\,\lvert \Omega_{T_{1}}\rvert,\,{\tilde C})\lVert \varphi\rVert_{L^{p}(0,\,T_{1};\,V_{0})},
\end{align*}
where (\ref{Eq (Section 3): Bounds of nabla E-epsilon}) is used.
This yields $\lVert \partial_{t}u_{\varepsilon}\rVert_{L^{p^{\prime}}(0,\,T;\,V_{0}^{\prime})}\le {\check C}$.
Therefore, ${\mathbf J}_{\varepsilon}$ is uniformly bounded for $\varepsilon\in(0,\,\varepsilon_{0})$.

Carrying out the standard weak compactness argument, we find a limit function $u_{0}\in u_{\star}+X_{0}^{p}(0,\,T;\,\Omega)$ such that
\begin{equation}\label{Eq (Section 3): Weak Conv of u}
  u_{\varepsilon_{j}}-u_{\star}\rightharpoonup u_{0}-u_{\star}\quad \textrm{in}\quad L^{p}(0,\,T_{1};\,V_{0})
\end{equation}
and
\begin{equation}\label{Eq (Section 3): Weak Conv of u derivative}
  \partial_{t}u_{\varepsilon}\rightharpoonup \partial_{t}u_{0}\quad \textrm{in}\quad L^{p^{\prime}}(0,\,T_{1};\,V_{0}^{\prime})
\end{equation}
where $\{\varepsilon_{j}\}_{j=0}^{\infty}\subset(0,\,1)$ is a decreasing sequence such that $\varepsilon_{j}\to 0$ as $j\to \infty$.
Also, the identity $u_{0}|_{t=0}=u_{\star}|_{t=0}$ in $L^{2}(\Omega)$ is straightforwardly shown (see \cite[Proposition 2.4]{T-parabolic}).
From (\ref{Eq (Section 3): Weak Conv of u}), we would like to prove 
\begin{equation}\label{Eq (Section 3): Strong Conv of u}
  \nabla u_{\varepsilon_{j}}\to \nabla u_{0}\quad \text{in}\quad L^{p}(\Omega_{T_{1}}).
\end{equation}
By (\ref{Eq (Section 3): Strong Monotonicity of E-epsilon}) and H\"{o}lder's inequality, we have
\begin{align*}
 & \iint_{\Omega_{T_{1}}}\lvert\nabla u_{\varepsilon_{j}}-\nabla u_{0}\rvert^{p}\,{\mathrm d}x{\mathrm d}t\\ 
 &\le \left(\iint_{\Omega_{T_{1}}}\left(\varepsilon_{j}^{2}+\lvert \nabla u_{\varepsilon_{j}}\rvert^{2}+\lvert\nabla u_{0}\rvert^{2} \right)^{p/2}\,{\mathrm d}x{\mathrm d}t \right)^{1-p/2}\\ 
 &\quad \cdot  \left(\iint_{\Omega_{T_{1}}}\left(\varepsilon_{j}^{2}+\lvert \nabla u_{\varepsilon_{j}}\rvert^{2}+\lvert\nabla u_{0}\rvert^{2} \right)^{p/2-1}\lvert \nabla u_{\varepsilon_{j}}-\nabla u_{0}\rvert^{2}\,{\mathrm d}x{\mathrm d}t\right)^{p/2}\\ 
 &\le C\left({\mathbf I}_{1,\,\varepsilon_{j}}+{\mathbf I}_{2,\,\varepsilon_{j}} \right)^{p/2},
\end{align*}
where 
\[
\begin{array}{rcl}
  {\mathbf I}_{1,\,\varepsilon_{j}} &\coloneqq &\displaystyle\iint_{\Omega_{T_{1}}}\left\langle \nabla E_{\varepsilon_{j}}(\nabla u_{\varepsilon})\mathrel{}\middle|\mathrel{}\nabla (u_{\varepsilon_{j}}-\nabla u_{0}) \right\rangle\,{\mathrm d}x{\mathrm d}t, \\ 
  {\mathbf I}_{2,\,\varepsilon_{j}} &\coloneqq &\displaystyle\iint_{\Omega_{T_{1}}}\left\langle \nabla E_{\varepsilon_{j}}(\nabla u_{0})\mathrel{}\middle|\mathrel{}\nabla (u_{\varepsilon_{j}}-\nabla u_{0}) \right\rangle\,{\mathrm d}x{\mathrm d}t.
\end{array}  
\]
For $\delta\in(0,\,T_{1}/2)$, which tends to $0$ later, we define a function $\phi_{\delta}\colon \lbrack 0,\,T_{1}\rbrack \to \lbrack 0,\,1\rbrack$ as
\begin{equation}\label{Eq (Section 3): Def of phi-delta}
  \phi_{\delta}(t) \coloneqq \left\{ \begin{array}{cc}  1 & (0\le t<T_{1}-\delta), \\ -\delta^{-1}(t-T_{1}) & (T_{1}-\delta\le t\le T_{1}).  \end{array}\right.
\end{equation}
We test $\varphi\coloneqq (u_{\varepsilon}-u_{0}) \phi_{\delta}$ into (\ref{Eq (Section 3): Approximate Equation}) with $\varepsilon=\varepsilon_{j}$, and integrate by parts.
Then, we have 
\begin{align*}
  &-\frac{1}{2}\iint_{\Omega_{T_{1}}}\lvert u_{\varepsilon_{j}}-u_{0} \rvert^{2}\partial_{t}\phi_{\delta}\,{\mathrm d}x{\mathrm d}t+\iint_{\Omega_{T_{1}}}\left\langle \nabla E_{\varepsilon_{j}}(\nabla u_{\varepsilon_{j}})\mathrel{}\middle|\mathrel{}\nabla (u_{\varepsilon_{j}}-u_{0}) \right\rangle \phi_{\delta}\,{\mathrm d}x{\mathrm d}t\\ 
  &=-\int_{0}^{T_{1}}\langle \partial_{t}u_{0},\, u_{\varepsilon_{j}}-u_{0}\rangle_{V_{0}^{\prime},\,V_{0}}\phi_{\delta}\,{\mathrm d}t \nonumber
\end{align*}
Discarding the first integral, and letting $\delta\to 0$, and then $j\to\infty$, we obtain 
\[\limsup_{j\to\infty} {\mathbf I}_{1,\,\varepsilon_{j}}\le \limsup_{j\to\infty}\left(-\int_{0}^{T_{1}}\langle\partial_{t}u_{0},\,(u_{\varepsilon_{j}}-u_{\star})-(u_{0}-u_{\star})  \rangle_{V_{0}^{\prime},\,V_{0}}\,{\mathrm d}t\right)=0,\]
where the last identity follows from (\ref{Eq (Section 3): Weak Conv of u}).
The strong convergence $\nabla E_{\varepsilon_{j}}(\nabla u_{0})\to A_{0}(\nabla u_{0})$ in $L^{p^{\prime}}(\Omega_{T_{1}})$ and the weak convergence $\nabla u_{\varepsilon_{j}}\rightharpoonup \nabla u_{0}$ in $L^{p}(\Omega_{T_{1}})$ follow from from Lemma \ref{Lemma: Fundamental Conv} \ref{Item 1/2} and (\ref{Eq (Section 3): Weak Conv of u}) respectively.
These convergence results yield ${\mathbf I}_{2,\,\varepsilon_{j}}\to 0$.
As a consequence, we have
\[\limsup_{j\to \infty}\iint_{\Omega_{T_{1}}}\lvert\nabla u_{\varepsilon_{j}}-\nabla u_{0}\rvert^{p}\,{\mathrm d}x{\mathrm d}t\le C\left(\sum_{l=1}^{2}\limsup_{j\to\infty}{\mathbf I}_{l,\,\varepsilon_{j}}\right)^{p/2}\le 0,\]
which completes the proof of (\ref{Eq (Section 3): Strong Conv of u}).
In particular, we may let $\nabla u_{\varepsilon_{j}}\to \nabla u_{0}$ a.e. in $\Omega_{T_{1}}$, by taking a subsequence if necessary.
Also, we are allowed to apply Lemma \ref{Lemma: Fundamental Conv} \ref{Item 2/2}.
From this and (\ref{Eq (Section 3): Weak Conv of u derivative}), we conclude that $u_{0}$ is a weak solution of (\ref{Eq (Section 3): Dirichlet Prob}).

The uniqueness of $u_{0}$ easily follows from monotone properties.
More precisely, letting $\varepsilon\to 0$ in (\ref{Eq (Section 3): Strong Monotonicity of E-p-epsilon}), we have
\[\langle \nabla E_{p}(z)-\nabla E_{p}(w)\mid z-w\rangle\ge \lambda \left(\lvert z\rvert^{2}+\lvert w\rvert^{2}\right)^{p/2-1}\lvert z-w\rvert^{2}>0\]
for all $z$, $w\in{\mathbb R}^{n}$ with $z\neq w$.
We also recall 
\[\langle \zeta_{1}-\zeta_{2}\mid z_{1}-z_{2}\rangle \ge 0\]
for all $z_{j}\in{\mathbb R}^{n}$, $\zeta_{j}\in \partial E_{1}(z_{j})$ with $j\in\{\,1,\,2\,\}$, which is often called the monotonicity of $\partial E_{1}$.
From these inequalities, we straightforwardly conclude that the weak solution of (\ref{Eq (Section 3): Dirichlet Prob}) is unique.
For the detailed discussions, see \cite[Proposition 2.4]{T-parabolic}.
\end{proof}

\section{Local bounds of solutions}\label{Sect:Bound of u}
In Section \ref{Sect:Bound of u}, we would like to show the local boundedness of $u$ and $u_{\varepsilon}$.
\subsection{Local $L^{\infty}$ estimate by Moser's iteration}\label{Subsect: Moser}
The local bound of $u$ follows from (\ref{Eq (Section 1): Higher Integrability}) (see also \cite[Theorem 2]{Choe MR1135917} and \cite[Appendix A]{MR2865434}) .
\begin{proposition}\label{Prop: L-infty bounds}
Under the assumptions in Theorem \ref{Thm}, we have 
\[
\lVert u\rVert_{L^{\infty}(Q_{R/2})}\le C(n,\,p,\,s,\,\lambda,\,\Lambda,\,K)\left(R^{-s_{\mathrm{c}}}\fiint_{Q_{R}}\left(\lvert u\rvert+1\right)^{s}\,{\mathrm d}x {\mathrm d}t\right)^{1/(s-s_{\mathrm c})}
\]
for any fixed $Q_{R}=Q_{R}(x_{0},\,t_{0})\Subset \Omega_{T}$ with $R\in(0,\,1)$.
\end{proposition}
\begin{proof}
We first prove a reversed H\"{o}lder estimate for $U\coloneqq \sqrt{1+\lvert u\rvert^{2}}$ .
More precisely, we claim that for any $\beta\in\lbrack s,\,\infty)$, and $r_{1},\,r_{2}\in(0,\,R\rbrack$ with $r_{1}<r_{2}$, there holds
\begin{equation}\label{Eq (Section 4): Reversed Hoelder for u)}
    \iint_{Q_{r_{1}}}U^{\kappa\beta+p-2}\,{\mathrm d}x{\mathrm d}t \le \left[\frac{C(\beta-1)^{\gamma}}{(r_{2}-r_{1})^{2}}\iint_{Q_{r_{2}}}U^{\beta} \,{\mathrm d}x{\mathrm d}t \right]^{\kappa},
\end{equation}
provided $U\in L^{\beta}(Q_{r_{2}})$.
Here $\kappa\coloneqq 1+p/n$, $\gamma\coloneqq p(1+1/(n+p))$ are fixed constants, and the constant $C\in(1,\,\infty)$ depends at most on $n$, $p$, $\lambda$, $\Lambda$, and $K$.
To prove (\ref{Eq (Section 4): Reversed Hoelder for u)}), we introduce a truncation parameter $M\in(1,\,\infty)$ .
For given $r_{1}$ and $r_{2}$, we choose and fix $\eta\in C_{\mathrm c}^{1}(B_{r_{2}}(x_{0});\,\lbrack 0,\,1\rbrack)$ and $\phi_{\mathrm c}\in C^{1}(\lbrack t_{0}-r_{2}^{2},\,t_{0}\rbrack;\,\lbrack 0,\,1\rbrack)$ satisfying
\begin{equation}\label{Eq (Section 4): Choice of Cut-Off}
  \eta|_{B_{r_{1}}}=1,\quad \phi_{\mathrm c}|_{I_{r_{1}}}=1,\quad \lVert \nabla\eta\rVert_{L^{\infty}(B_{r_{2}})}^{2}+\lVert \partial_{t}\phi_{\mathrm c}\rVert_{L^{\infty}(I_{r_{2}})}\le \frac{c_{0}}{(r_{2}-r_{1})^{2}}, 
\end{equation}
and $\phi_{\mathrm c}(t_{0}-r_{2}^{2})=0$, where $c_{0}\in(1,\,\infty)$ is a universal constant. 
We let $\phi_{\mathrm h}\colon \lbrack t_{0}-r_{2}^{2},\,t_{0}\rbrack \to\lbrack 0,\,1\rbrack$ be a non-increasing Lipschitz function satisfying $\phi_{\mathrm h}(t_{0}-r_{2}^{2})=1$ and $\phi_{\mathrm h}(t_{0})=0$, and we write $\phi\coloneqq \phi_{\mathrm c}\phi_{\mathrm h}$.
Thanks to the Steklov average method, we may test $\varphi\coloneqq \eta^{p}\phi\psi_{\alpha,\,M}(U )u$ into (\ref{Eq (Section 1): Weak form in def}), where $\psi_{\alpha,\,M}$ is defined as (\ref{Eq (Section 2): psi-alpha-M}) with $\alpha\coloneqq \beta-2>0$.
Then, we obtain
\begin{align*}
&-\iint_{Q_{r}}\eta^{p}\Psi_{\alpha,\,M}(U )\partial_{t}\phi \,{\mathrm d}x {\mathrm d}t+\iint_{Q_{r}}\left\langle Z+\nabla E_{p}(\nabla u)\mathrel{}\middle|\mathrel{} \nabla u \right\rangle\eta^{p}\psi_{\alpha,\,M}(U ) \phi\,{\mathrm d}x {\mathrm d}t\\&\quad +\iint_{Q_{r}}\left\langle Z+\nabla E_{p}(\nabla u)\mathrel{}\middle|\mathrel{} \nabla u \right\rangle\eta^{p}\psi_{\alpha,\,M}^{\prime}(U )\lvert u\rvert^{2}U ^{-1} \phi\,{\mathrm d}x {\mathrm d}t\\ &=-\iint_{Q_{r}}\left\langle Z+\nabla E_{p}(\nabla u)\mathrel{}\middle|\mathrel{} \nabla \eta \right\rangle\eta^{p-1}\psi_{\alpha,\,M}(U )u \phi\,{\mathrm d}x {\mathrm d}t\\ &\le C\iint_{Q_{r}}\left(1+\lvert\nabla u\rvert^{p-1} \right)\lvert\nabla\eta\rvert\eta^{p-1}\psi_{\alpha,\,M}(U )\lvert u\rvert\phi\,{\mathrm d}x {\mathrm d}t\\ &\le \frac{\lambda}{2}\iint_{Q_{r}}\eta^{p}\lvert\nabla u\rvert^{p}\psi_{\alpha,\,M}(U )\phi\,{\mathrm d}x {\mathrm d}t+C\iint_{Q_{r}}\left(\lvert\nabla\eta\rvert^{p}\lvert u\rvert^{p}+\eta^{p}\right) \psi_{\alpha,\,M}(U )\phi\,{\mathrm d}x {\mathrm d}t.
\end{align*}
By (\ref{Eq (Section 1): Differentiated Eq}), (\ref{Eq (Section 2): Euler's id})--(\ref{Eq (Section 2): Bound of Z}), and (\ref{Eq (Section 3): Strong Monotonicity of E-p}), we get
\begin{align*}
&-\iint_{Q_{r}} \eta^{p}\Psi_{\alpha,\,M}(U )\phi_{\mathrm c} \partial_{t}\phi_{\mathrm h}\,{\mathrm d}x {\mathrm d}t+\frac{c_{1}}{2}\iint_{Q_{r}}\eta^{p}\lvert\nabla u\rvert^{p}\psi_{\alpha,\,M}(U )\phi_{\mathrm c}\phi_{\mathrm h}\,{\mathrm d}x {\mathrm d}t\\
&\le \iint_{Q_{r}} \eta^{p}\Psi_{\alpha,\,M}(U )\phi_{\mathrm h} \partial_{t}\phi_{\mathrm c}\,{\mathrm d}x {\mathrm d}t+C\iint_{Q_{r}}\left(\lvert\nabla\eta\rvert^{p}U ^{p}+\eta^{p}\right) \psi_{\alpha,\,M}(U )\,{\mathrm d}x {\mathrm d}t.
\end{align*}
Choosing $\phi_{\mathrm h}$ suitably and recalling our choice of $\eta$ and $\phi_{\mathrm c}$, we easily obtain
\begin{align*}
    &\esssup_{t_{0}-r_{2}^{2}<t<t_{0}}\int_{B_{r_{2}}}\eta^{p}\phi_{\mathrm c}\Psi_{\alpha,\,M}(U ) \,{\mathrm d}x+\iint_{Q_{r_{2}}}\eta^{p}\phi_{\mathrm c}\lvert\nabla u\rvert^{p}\psi_{\alpha,\,M}(U )\,{\mathrm d}x{\mathrm d}t\nonumber\\ 
&\le C\left[\iint_{Q_{r_{2}}}\left(\frac{\psi_{\alpha,\,M}(U )U ^{2}}{(r_{2}-r_{1})^{2}}+\frac{\psi_{\alpha,\,M}(U )U ^{p}}{(r_{2}-r_{1})^{p}}\right)\,{\mathrm d}x{\mathrm d}t \right]
\end{align*}
Also, it is easy to deduce
\begin{align*}
    &\iint_{Q_{r_{1}}}\Psi_{\alpha,\,M}(U )^{p/n}\psi_{\alpha,\,M}(U )U ^{p} \,{\mathrm d}x{\mathrm d}t\\ 
    &\le C_{n,\,p}\left(\esssup_{t_{0}-r_{2}^{2}<t<t_{0}}\int_{B_{r_{2}}}\eta^{p}\Psi_{\alpha,\,M}(U ) \phi_{\mathrm c}\,{\mathrm d}x \right)^{p/n}\iint_{Q_{r_{2}}}\left\lvert\nabla \left(\eta \psi_{\alpha,\,M}(U )^{1/p}U \right)  \right\rvert^{p}\phi_{\mathrm c} \,{\mathrm d}x{\mathrm d}t
  \end{align*}
by H\"{o}lder's inequality and the Sobolev embedding $W_{0}^{1,\,p}(B_{r_{2}})\hookrightarrow L^{\frac{np}{n-p}}(B_{r_{2}})$.
By direct computations, we notice
\begin{align*}
    &\iint_{Q_{r_{2}}}\left\lvert\nabla \left(\eta \psi_{\alpha,\,M}(U )^{1/p}U \right)  \right\rvert^{p}\phi_{\mathrm c} \,{\mathrm d}x{\mathrm d}t\\ 
    &\le C_{p}\left[\iint_{Q_{r_{2}}}\lvert \nabla\eta\rvert^{p}\psi_{\alpha,\,M}(U )U ^{p}\phi_{\mathrm c}\,{\mathrm d}x {\mathrm d}t+(1+\alpha)^{p}\iint_{Q_{r_{2}}}\eta^{p}\lvert \nabla u\rvert^{p}\psi_{\alpha,\,M}(U )\phi_{\mathrm c}\,{\mathrm d}x {\mathrm d}t \right],
\end{align*}
where we have used (\ref{Eq (Section 2): psi estimate 0}) with $r=p$, and $\lvert \nabla U \rvert\le \lvert \nabla u\rvert$.
Combining these three estimates, we get
\[\iint_{Q_{r_{1}}}\Psi_{\alpha,\,M}(U )^{p/n}\psi_{\alpha,\,M}(U )U ^{p} \,{\mathrm d}x{\mathrm d}t\le \left[\frac{C(\alpha+1)^{p}}{(r_{2}-r_{1})^{2}} \iint_{Q_{r_{2}}}\left(U ^{\beta}+1\right)\,{\mathrm d}x{\mathrm d}t \right]^{\kappa},\]
where Young's inequality is used.
Letting $M\to \infty$ and making use of Beppo Levi's monotone convergence theorem and (\ref{Eq (Section 2): Monotone conv of psi-s}), we conclude (\ref{Eq (Section 4): Reversed Hoelder for u)}). 

We define the sequences $\{R_{l}\}_{l=0}^{\infty}\subset (R/2,\,R\rbrack$, $\{q_{l}\}_{l=0}^{\infty}\subset \lbrack s,\,\infty)$, and $\{Y_{l}\}_{l=0}^{\infty}\subset {\mathbb R}_{\ge 0}$ as
\[R_{l}\coloneqq \frac{1+2^{-l}}{2}R,\quad q_{l}\coloneqq \kappa^{l}\mu +s_{\mathrm{c}},\quad Y_{l}\coloneqq\left(\iint_{Q_{R_{l}}}U ^{q_{l}}\,{\mathrm d}x {\mathrm d}t \right)^{1/q_{l}},\]
where $\mu\coloneqq s-s_{\mathrm c}\in{\mathbb R}_{>0}$.
Then, by $(q_{l}-1)^{\gamma}=\mu^{\gamma}\left(\kappa^{l}+(s_{\mathrm c}-1)/\mu\right)^{\gamma}$, it is easy to check that
\begin{equation}\label{Eq (Section 4): Iteration on exponents}
  R_{l}-R_{l+1}=2^{-l-2}R\quad \text{and}\quad (q_{l}-1)^{\gamma}\le (2\mu)^{\gamma}{\tilde\kappa}^{\gamma l}
\end{equation}
hold for every $l\in{\mathbb Z}_{\ge 0}$, where the constant ${\tilde\kappa}\in(\kappa,\,\infty)$ depends at most on $\kappa$, $s$, and $s_{\mathrm c}$.
Hence, (\ref{Eq (Section 4): Reversed Hoelder for u)}) with $\beta\coloneqq q_{l}\ge q_{0}=s>s_{\mathrm c}\ge 2$ yields
\(Y_{l+1}^{q_{l+1}}\le (AB^{l}Y_{l}^{q_{l}})^{\kappa}
\) for all $l\in{\mathbb Z}_{\ge 0}$, where $A\coloneqq CR^{-2}\in(1,\,\infty)$ for some constant $C\in(1,\,\infty)$, and $B\coloneqq 4{\tilde\kappa}^{\gamma}\in(1,\,\infty)$. 
By Lemma \ref{Lemma: Moser iteration lemma}, we have
\[\esssup_{Q_{R/2}} U\le\limsup_{l\to\infty}Y_{l}\le C(n,\,p,\,\lambda,\,\Lambda,\,K)\left(R^{-s_{\mathrm{c}}}\fiint_{Q_{R}}U ^{s}\,{\mathrm d}x {\mathrm d}t \right)^{1/\mu},\]
which completes the proof.
\end{proof}

\subsection{Comparison principle for parabolic approximate equations}\label{Subsect: CP}
We would like to show the comparison principle and the weak maximum principle for (\ref{Eq (Section 3): Approximate Equation}).
The weak maximum principle implies that an approximate solution $u_{\varepsilon}$ will be bounded if it admits a Dirichlet boundary datum in $L^{\infty}$.
Hence, combining with Proposition \ref{Prop: L-infty bounds}, we may let $u_{\varepsilon}$ be locally bounded, which is used in Section \ref{Sect: L-p to L-infty}.

We straightforwardly fix some terminology.
\begin{definition}\label{Def: Sub-Sup}\upshape
Let $u$, $v\in X^{p}(0,\,T;\,\Omega)\cap C^{0}(\lbrack 0,\,T\rbrack;\,L^{2}(\Omega))$. 
\begin{enumerate}
\item It is said that $u\le v$ on $\partial_{\mathrm{p}}\Omega_{T}$, when there hold $(u-v)_{+}\in X_{0}^{p}(0,\,T;\,\Omega)$ and $(u-v)_{+}|_{t=0}=0$ in $L^{2}(\Omega)$.
\item A function $u$ is called a weak subsolution to (\ref{Eq (Section 3): Approximate Equation}) in $\Omega_{T}$ when
\begin{equation}\label{Eq (Section 4): subsol}
\int_{0}^{T} \langle \partial_{t}u,\, \varphi\rangle_{{V_{0}^{\prime}},\,V_{0}}\,{\mathrm d}t+\iint_{\Omega_{T}}\left\langle \nabla E^{\varepsilon}  (\nabla u)\mathrel{}\middle|\mathrel{} \nabla\varphi \right\rangle\,{\mathrm d}x {\mathrm d}t\le 0
\end{equation}
holds for all non-negative $\varphi\in X_{0}^{p}(0,\,T;\,\Omega)$.
\item A function $v$ is called a weak supersolution to (\ref{Eq (Section 3): Approximate Equation}) in $\Omega_{T}$ when
\begin{equation}\label{Eq (Section 4): supersol}
\int_{0}^{T} \langle \partial_{t}v,\, \varphi\rangle_{{V_{0}^{\prime}},\,V_{0}}\,{\mathrm d}t+\iint_{\Omega_{T}}\left\langle \nabla E^{\varepsilon}  (\nabla v)\mathrel{}\middle|\mathrel{} \nabla\varphi \right\rangle\,{\mathrm d}x {\mathrm d}t\ge 0
\end{equation}
holds for all non-negative $\varphi\in X_{0}^{p}(0,\,T;\,\Omega)$.
\end{enumerate}
\end{definition}

\begin{proposition}\label{Prop: CP}
Let $u$ and $v$ be respectively a subsolution and a supersolution to (\ref{Eq (Section 3): Approximate Equation}).
If $u\le v$ on $\partial_{\mathrm{p}}\Omega_{T}$ holds in the sense of Definition \ref{Def: Sub-Sup}, then $u\le v$ a.e.~in $\Omega_{T}$. 
\end{proposition}
\begin{proof}
For $\delta\in(0,\,T/2)$, we define $\phi_{\delta}$ as (\ref{Eq (Section 3): Def of phi-delta}) with $T_{1}$ replaced by $T$.
Thanks to the Steklov average, we may test $(u-v)_{+}\phi_{\delta}$ into (\ref{Eq (Section 4): subsol})--(\ref{Eq (Section 4): supersol}).
Integrating by parts, we have
\[0\ge -\iint_{\Omega_{T}}\lvert u-v\rvert^{2}\partial_{t}\phi_{\delta}\,{\mathrm d}x{\mathrm d}t+\iint_{\Omega_{T}}\left\langle \nabla E^{\varepsilon}  (\nabla u)-\nabla E^{\varepsilon}  (\nabla v)\mathrel{} \middle|\mathrel{}\nabla (u-v)_{+} \right\rangle\phi_{\delta}\,{\mathrm d}x {\mathrm d}t.\]
Discarding the first integral and letting $\delta \to 0$, we obtain 
\[0 \ge \iint_{\Omega_{T}}\left\langle \nabla E^{\varepsilon}  (\nabla u)-\nabla E^{\varepsilon}  (\nabla v)\mathrel{} \middle|\mathrel{}\nabla u-\nabla v \right\rangle\chi_{\{u>v\}} \,{\mathrm d}x {\mathrm d}t\]
by Beppo Levi's monotone convergence theorem.
Since the mapping ${\mathbb R}^{n}\ni z\mapsto \nabla E^{\varepsilon}  (z)\in{\mathbb R}^{n}$ is strictly monotone, the inequality above yields $\nabla (u-v)_{+}=0$ a.e. in $\Omega_{T}$. 
Recalling $(u-v)_{+}\in L^{p}(0,\,T;\,W_{0}^{1,\,p}(\Omega))$, we have $u\le v$ a.e. in $\Omega_{T}$.
\end{proof}

As a consequence, we can deduce the following Corollary \ref{Lemma: WMP}.
\begin{corollary}\label{Lemma: WMP}
Let $u_{\star}\in L^{\infty}(\Omega_{T})\cap X_{0}^{p}(0,\,T;\,\Omega)$, and assume that $u_{\varepsilon}\in u_{\star}+X_{0}^{p}(0,\,T;\,\Omega)$ is the weak solution of (\ref{Eq (Section 3): Approximate Equation}).
Then, $u_{\varepsilon}\in L^{\infty}(\Omega_{T})$, and  
\[\lVert u_{\varepsilon}\rVert_{L^{\infty}(\Omega_{T})} \le\lVert u_{\star}\rVert_{L^{\infty}(\Omega_{T})}.\]
\end{corollary}
\begin{proof}
We abbrebiate $M\coloneqq \lVert u_{\star}\rVert_{L^{\infty}(\Omega_{T})}\in\lbrack 0,\,\infty)$.
It is clear that the constant functions $\pm M$ are weak solutions to (\ref{Eq (Section 3): Approximate Equation}).
Since $0\le (u_{\varepsilon}-M)_{+}\le (u_{\varepsilon}-u_{\star})_{+}$ holds a.e.~in $\Omega_{T}$, it is easy to check $u_{\varepsilon}\le M$ on $\partial_{\mathrm{p}}\Omega_{T}$ in the sense of Definition \ref{Def: Sub-Sup} (see \cite[Lemma 1.25]{MR1207810}).
Similarly, there holds $-M\le u_{\varepsilon}$ on $\partial_{\mathrm{p}}\Omega_{T}$.
By Proposition \ref{Prop: CP}, we have $-M\le u_{\varepsilon}\le M$ in $\Omega_{T}$, which completes the proof.
\end{proof}
\begin{remark}\upshape
The proofs of Propositions \ref{Prop: Convergence}, \ref{Prop: CP}, and Corollary \ref{Lemma: WMP} work even when the domain $\Omega_{T}=\Omega\times (0,\,T)$ is replaced by its parabolic subcylinder $Q_{R}(x_{0},\,t_{0})=B_{R}(x_{0})\times (t_{0}-R^{2},\,t_{0}\rbrack\Subset \Omega_{T}$.
Therefore, we can apply these results with $\Omega_{T}$ replaced by $Q_{R}(x_{0},\,t_{0})$, which are to be used in the proof of Theorem \ref{Thm}.
\end{remark}

\section{Regularity for gradients of approximate solutions}\label{Sect: L-p to L-infty}
In Section \ref{Sect: L-p to L-infty}, we consider a bounded weak solution to (\ref{Eq (Section 3): Approximate Equation}) in a parabolic subcylinder $\tilde{Q}\Subset \Omega_{T}$, and fix $Q_{R}(x_{0},\,t_{0})=B_{R}(x_{0})\times (t_{0}-R^{2},\,t_{0}\rbrack \Subset {\tilde Q}$.
Throughout Sections \ref{Sect: L-p to L-infty}--\ref{Sect:MainTheorem}, we assume 
\begin{equation}\label{Eq (Section 5): Bounds of u-epsilon}
\esssup_{Q_{R}(x_{0},\,t_{0})}\,\lvert u_{\varepsilon}\rvert\le M_{0}
\end{equation}
for some $M_{0}\in(0,\,\infty)$.
We aim to prove that a gradient $\nabla u_{\varepsilon}$ is locally in $L^{q}$ for any $q\in(p,\,\infty\rbrack$, whose estimate may depend on $M_{0}$ but is independent of $\varepsilon\in(0,\,1)$.

\subsection{Weak formulations and energy estimates}\label{Subsect: Energy}
As well as $V_{\varepsilon}\coloneqq \sqrt{\varepsilon^{2}+\lvert \nabla u_{\varepsilon}\rvert^{2}}$, we also consider another function $W_{\varepsilon}$, defined as 
\[W_{\varepsilon}\coloneqq \sqrt{1+\sum_{j=1}^{n}w_{\varepsilon,\,j}^{2}}\le 1+V_{\varepsilon},\]
where for each $j\in\{\,1,\,\dots\,,\,n\,\}$, we set
\[w_{\varepsilon,\,j}\coloneqq (\partial_{x_{j}}u_{\varepsilon}-1)_{+}-(-\partial_{x_{j}}u_{\varepsilon}-1)_{+}.\]
We note that $V_{\varepsilon}$ and $W_{\varepsilon}$ are compatible, in the sense that there hold
\begin{equation}\label{Eq (Section 5): Compatibility}
  \left\{\begin{array}{rcl}
    V_{\varepsilon} \le c_{n}W_{\varepsilon}\le c_{n}(1+V_{\varepsilon}) & \textrm{in} & Q=Q_{R},\\
    W_{\varepsilon}\le \sqrt{2} V_{\varepsilon} & \textrm{in} &D\subset Q_{R}, 
  \end{array}  \right.
\end{equation}
where $D\coloneqq \{Q_{R}\mid \lvert\nabla u_{\varepsilon}\rvert>1\}$ (see \cite[\S 4.1]{T-Lipschitz MR4201656}).
In particular, we are allowed to use
\begin{equation}\label{Eq (Section 5): Uniform elliptic structure on D}
 {\tilde \lambda}W_{\varepsilon}^{p-2}\mathrm{id}_{n} \leqslant \nabla^{2}E^{\varepsilon}  (\nabla u_{\varepsilon}) \leqslant {\tilde \Lambda}W_{\varepsilon}^{p-2}\mathrm{id}_{n} \quad \textrm{in } D,
\end{equation}
where $\hat{\lambda}=\hat{\lambda}(n,\,p,\,\lambda)\in(0,\,\lambda)$ and $\hat{\Lambda}=\hat{\Lambda}(n,\,p,\,\Lambda,\,K)\in (\Lambda,\,\infty)$ are constants.

Combining with (\ref{Eq (Section 5): Bounds of u-epsilon}), we can carry out standard parabolic arguments, including the difference quotient method, Moser's iteration, and De Giorgi's truncation (see \cite[Chapter VIII]{MR1230384}).
As a consequence, we are allowed to let
\begin{equation}\label{Eq (Section 5): Improved reg}
\nabla u_{\varepsilon}\in L^{\infty}(Q_{R};\,{\mathbb R}^{n}) \quad \text{and} \quad\nabla^{2}u_{\varepsilon}\in L^{2}(Q_{R};\,{\mathbb R}^{n\times n}).
\end{equation}
Thanks to this improved regularity, there holds
\begin{equation}\label{Eq (Section 5): Weak form differentiated}
-\iint_{Q_{R}} \partial_{x_{j}}u_{\varepsilon}\partial_{t}\varphi\,{\mathrm d}x{\mathrm d}t +\iint_{Q_{R}}\left\langle \nabla^{2}E^{\varepsilon}  (\nabla u_{\varepsilon})\nabla\partial_{x_{j}}u_{\varepsilon}\mathrel{}\middle|\mathrel{}\nabla\varphi  \right\rangle\,{\mathrm d}x{\mathrm d}t=0
\end{equation}
for any $\varphi\in C_{\mathrm c}^{1}(Q_{R})$.
Moreover, we may extend the test function $\varphi$ in the class
\(X_{0}^{2}(I_{R};\,B_{R})\coloneqq \left\{ \varphi\in L^{2}(I_{R};\,W_{0}^{1,\,2}(B_{R}))\mathrel{}\middle|\mathrel{}\partial_{t}\varphi\in L^{2}(I_{R};\,W^{-1,\,2}(B_{R}))\right\}\subset C(\overline{I_{R}};\,L^{2}(B_{R}))\)
with $\varphi|_{t=t_{0}-r^{2}}=\varphi|_{t=t_{0}}=0$ in $L^{2}(B_{R})$.
From (\ref{Eq (Section 5): Weak form differentiated}), we deduce basic energy estimates concerning $V_{\varepsilon}$ and $W_{\varepsilon}$ (Lemma \ref{Lemma: Energy estimates}).
\begin{lemma}\label{Lemma: Energy estimates}
For $\alpha\in\lbrack 0,\,\infty)$, $M\in(1,\,\infty)$, let $\psi_{\alpha,\,M}$ and ${\tilde \psi}_{\alpha,\,M}$ be given by (\ref{Eq (Section 2): psi-alpha-M}) and (\ref{Eq (Section 2): psi-alpha-M tilde}) respectively.
Let $u_{\varepsilon}$ be a weak solution to (\ref{Eq (Section 3): Approximate Equation}) in ${\tilde Q}$. Fix $Q_{R}(x_{0},\,t_{0})\Subset {\tilde Q}$, and let (\ref{Eq (Section 5): Improved reg}) be in force.
Fix $\eta\in C_{\mathrm c}^{1}(B_{R}(x_{0});\,\lbrack 0,\,1\rbrack)$, and $\phi_{\mathrm{c}}\in C^{1}(\lbrack t_{0}-R^{2},\,t_{0}\rbrack;\,\lbrack 0,\,1\rbrack)$ that satisfies $\phi_{\mathrm c}(t_{0}-R^{2})=0$.
Then, there hold
\begin{align}
  &
  \iint_{Q_{R}}V_{\varepsilon}^{p-2}\left(\lvert \nabla^{2} u_{\varepsilon}\rvert^{2}{\tilde \psi}_{\alpha,\,M}(V_{\varepsilon})+\lvert\nabla V_{\varepsilon}\rvert^{2}{\tilde\psi}_{\alpha,\,M}^{\prime}(V_{\varepsilon})V_{\varepsilon} \right)\eta^{2}\phi_{\mathrm{c}}\,{\mathrm d}x{\mathrm d}t\nonumber \\   
  &\le C\iint_{Q_{R}}\left(V_{\varepsilon}^{p}\lvert \nabla\eta\rvert^{2}+V_{\varepsilon}^{2}\lvert \partial_{t}\phi_{\mathrm{c}} \rvert\right) {\tilde \psi}_{\alpha,\,M}(V_{\varepsilon})\,{\mathrm d}x{\mathrm d}t,\label{Eq (Section 5): Energy estimate 1}
\end{align}
and
\begin{align}
  & \esssup_{\tau \in I_{R}}\int_{B_{R}\times \{\tau\}}{\Psi}_{\alpha,\,M}(W_{\varepsilon})\eta^{2}\phi_{\mathrm c}\,{\mathrm d}x\nonumber\\ 
  &\quad +\iint_{Q_{R}}W_{\varepsilon}^{p-2}\left(\psi_{\alpha,\,M}(V_{\varepsilon})+\psi_{\alpha,\,M}^{\prime}(W_{\varepsilon})W_{\varepsilon} \right)\eta^{2}\phi_{\mathrm{c}}\,{\mathrm d}x{\mathrm d}t\nonumber \\   
  &\le C\iint_{Q_{R}}\left(W_{\varepsilon}^{p}\lvert \nabla\eta\rvert^{2}+W_{\varepsilon}^{2}\lvert \partial_{t}\phi_{\mathrm{c}} \rvert\right) \psi_{\alpha,\,M}(W_{\varepsilon})\,{\mathrm d}x{\mathrm d}t,\label{Eq (Section 5): Energy estimate 2}
\end{align}
where $C\in(1,\,\infty)$ depends at most on $n$, $p$, $\lambda$, $\Lambda$, $K$.
\end{lemma}
\begin{proof}
Let $\zeta\in C_{\mathrm c}^{1}(Q_{R})$ be non-negative, and assume that a composite function $\psi\colon{\mathbb R}_{\ge 0}\to{\mathbb R}_{\ge 0}$ satisfy all the conditions given in Section \ref{Subsect: psi}.
We set $\phi\coloneqq \phi_{\mathrm{c}}\phi_{\mathrm{h}}$, where we choose an arbitrary Lipschitz function $\phi_{\mathrm{h}}\colon \lbrack t_{0}-R^{2},\,t_{0} \rbrack\to  \lbrack 0,\,1\rbrack$ that is non-increasing and satisfies $\phi_{\mathrm{h}}(t_{0})=0$.

To prove (\ref{Eq (Section 5): Energy estimate 1}), we test $\varphi=\zeta\psi(V_{\varepsilon})\partial_{x_{j}}u_{\varepsilon}$ into (\ref{Eq (Section 5): Weak form differentiated}).
This function is admissible by the method of Steklov averages.
Summing over $j\in \{\,1,\,\dots\,,\,n\,\}$, we have
\begin{align}\label{Eq (Section 5): Weak Form}
&-\iint_{Q_{R}}\Psi(V_{\varepsilon})\partial_{t}\zeta\,{\mathrm d}x{\mathrm d}t+\iint_{Q_{R}}\left\langle A_{\varepsilon}\nabla[\Psi(V_{\varepsilon})]\mathrel{}\middle|\mathrel{}\nabla\zeta \right\rangle\,{\mathrm d}x {\mathrm d}t\nonumber \\ &+\iint_{Q_{R}}\left[\left\langle A_{\varepsilon}\nabla V_{\varepsilon} \mathrel{}\middle|\mathrel{}\nabla V_{\varepsilon}  \right\rangle\psi^{\prime}(V_{\varepsilon})V_{\varepsilon}+\sum_{j=1}^{n}\left\langle  A_{\varepsilon}\nabla \partial_{x_{j}}u_{\varepsilon} \mathrel{}\middle|\mathrel{}\nabla \partial_{x_{j}}u_{\varepsilon}  \right\rangle\psi(V_{\varepsilon})\right]\zeta  \,{\mathrm d}x {\mathrm d}t\nonumber\\ &=0,
\end{align}
where $A_{\varepsilon}\coloneqq \nabla^{2}E^{\varepsilon}  (\nabla u_{\varepsilon})$, and $\Psi$ is defined as (\ref{Eq (Section 2): Convex Psi}).
Here we choose $\psi\coloneqq {\tilde \psi}_{\alpha,\,M}$ and $\zeta\coloneqq \eta^{2}\phi$.
Then, (\ref{Eq (Section 5): Weak Form}) yields
\begin{align*}
  &-\iint_{Q_{R}}{\tilde \Psi}_{\alpha,\,M}(V_{\varepsilon})\eta^{2}\phi_{\mathrm c}\partial_{t}\phi_{\mathrm h}\,{\mathrm d}x{\mathrm d}t\\
  & +\lambda\iint_{Q_{R}}V_{\varepsilon}^{p-2}\left({\tilde \psi}_{\alpha,\,M}(V_{\varepsilon})\lvert \nabla^{2}u_{\varepsilon}\rvert^{2} +{\tilde \psi}_{\alpha,\,M}^{\prime}(V_{\varepsilon})V_{\varepsilon}\lvert\nabla V_{\varepsilon}\rvert^{2} \right)\eta^{2}\phi_{\mathrm c}\phi_{\mathrm h}\,{\mathrm d}x{\mathrm d}t\\ 
  &\le \iint_{Q_{R}}{\tilde \Psi}_{\alpha,\,M}(V_{\varepsilon})\eta^{2}\phi_{\mathrm h}\partial_{t}\phi_{\mathrm c}\,{\mathrm d}x{\mathrm d}t+2(\Lambda+K)\iint_{Q_{R}}V_{\varepsilon}^{p-1}{\tilde \psi}_{\alpha,\,M}(V_{\varepsilon})\lvert \nabla V_{\varepsilon}\rvert\lvert\nabla\eta\rvert\eta\phi\,{\mathrm d}x\\ 
  &\le \frac{\lambda}{2}\iint_{Q_{R}}V_{\varepsilon}^{p-2}{\tilde \psi}_{\alpha,\,M}(V_{\varepsilon})\lvert \nabla^{2}u_{\varepsilon}\rvert^{2}\eta^{2}\phi_{\mathrm c}\phi_{\mathrm h}\,{\mathrm d}x{\mathrm d}t  \\ 
  &\quad +\frac{2(\Lambda+K)^{2}}{\lambda}\iint_{Q_{R}} V_{\varepsilon}^{p} {\tilde \psi}_{\alpha,\,M}(V_{\varepsilon})\lvert\nabla\eta\rvert^{2}\,{\mathrm d}x{\mathrm d}t+\iint_{Q_{R}}V_{\varepsilon}^{2}{\tilde \psi}_{\alpha,\,M}(V_{\varepsilon})\lvert\partial_{t}\phi_{\mathrm c}\rvert\,{\mathrm d}x{\mathrm d}t,
\end{align*}
where we have used $\lvert \nabla V_{\varepsilon}\rvert \le \lvert \nabla^{2}u_{\varepsilon}\rvert$, Young's inequality, and (\ref{Eq (Section 2): Psi-bounds}).
Discarding the first integral, and choosing $\phi_{\mathrm h}$ suitably, we easily conclude (\ref{Eq (Section 5): Energy estimate 1}).

To prove (\ref{Eq (Section 5): Energy estimate 2}), we test $\varphi=\zeta\psi(W_{\varepsilon})w_{\varepsilon,\,j}$ into (\ref{Eq (Section 5): Weak form differentiated}), where we let $\psi\coloneqq \psi_{\alpha,\,M}$ and $\zeta\coloneqq \eta^{2}\phi$.
Since all the integrands range over $\{\lvert \partial_{x_{j}}u_{\varepsilon}\rvert>1\}\subset D$, and therefore we may replace $\nabla\partial_{x_{j}}u_{\varepsilon}$ by $\nabla w_{\varepsilon,\,j}$, and apply (\ref{Eq (Section 5): Uniform elliptic structure on D}).
By similar computations, we have 
\begin{align*}
  &-\iint_{Q_{R}}\Psi_{\alpha,\,M}(W_{\varepsilon})\eta^{2}\phi_{\mathrm c} \partial_{t}\phi_{\mathrm h}\,{\mathrm d}x{\mathrm d}t\\ 
  &\quad +\frac{{\hat\lambda}}{2}\iint_{Q_{R}}W_{\varepsilon}^{p-2}\left(\psi_{\alpha,\,M}(W_{\varepsilon})+\psi_{\alpha,\,M}^{\prime}(W_{\varepsilon})W_{\varepsilon}  \right)\lvert \nabla W_{\varepsilon} \rvert^{2}\,\eta^{2}\phi_{\mathrm c}\phi_{\mathrm h}\,{\mathrm d}x{\mathrm d}t \\ 
  &\le \frac{2{\hat \Lambda}^{2}}{{\hat \lambda}}\iint_{Q_{R}}W_{\varepsilon}^{p}\psi_{\alpha,\,M}(W_{\varepsilon})\lvert\nabla \eta\rvert^{2}\,{\mathrm d}x{\mathrm d}t+\iint_{Q_{R}}V_{\varepsilon}^{2}\psi_{\alpha,\,M}(V_{\varepsilon})\lvert\partial_{t}\phi_{\mathrm c}\rvert\,{\mathrm d}x{\mathrm d}t.
\end{align*}
Choosing $\phi_{\mathrm h}$ suitably, we obtain (\ref{Eq (Section 5): Energy estimate 2}).
\end{proof}

\subsection{Reversed H\"{o}lder inequalities and local gradient bounds}\label{Sect:GradientBound}
We would like to show $L^{q}$-bounds of $\nabla u_{\varepsilon}$ for each $q\in(p,\,\infty\rbrack$, whose estimate depends on $M_{0}$ but is uniformly for $\varepsilon\in(0,\,1)$.

The case $q\in(p,\,\infty)$ is completed by (\ref{Eq (Section 5): Bounds of u-epsilon}) and (\ref{Eq (Section 5): Energy estimate 1}).
\begin{proposition}\label{Prop: Higher integrability}
Let $n$ and $p$ satisfy (\ref{Eq (Section 1): Subcritical Range}).
Assume that $u_{\varepsilon}$ is a weak solution to (\ref{Eq (Section 3): Approximate Equation}) in ${\tilde Q}$. 
Fix $Q_{R}(x_{0},\,t_{0})\Subset {\tilde Q}$ with $R\in(0,\,1)$, and let (\ref{Eq (Section 5): Bounds of u-epsilon}) and (\ref{Eq (Section 5): Improved reg}) be in force.
Then for each fixed $q\in(p,\,\infty)$, there holds 
\begin{equation}\label{Eq (Section 5): Reversed Hoelder}
\iint_{Q_{R/2}} V_{\varepsilon}^{q}\,{\mathrm d}x{\mathrm d}t \le C\iint_{Q_{R}}\left(V_{\varepsilon}^{p}+1\right)\,{\mathrm d}x{\mathrm d}t,
\end{equation}
where $C\in(1,\,\infty)$ depends at most on $n$, $p$, $q$, $\lambda$, $\Lambda$, $K$, $M_{0}$, and $r$.
\end{proposition}
\begin{proof}
It suffices to prove
\begin{equation}\label{Eq (Section 5): Reversed Hoelder Claim}
\iint_{Q_{r_{1}}}{\tilde \psi}_{\alpha,\,M}(V_{\varepsilon})V_{\varepsilon}^{2}\,{\mathrm d}x{\mathrm d}t\le \frac{C(n,\,p,\,\lambda,\,\Lambda,\,K,\,\alpha,\,M_{0})}{(r_{2}-r_{1})^{2}}\iint_{Q_{r_{2}}}\left(V_{\varepsilon}^{\alpha+p}+1\right)\,{\mathrm d}x{\mathrm d}t
\end{equation}
for $\alpha\in\lbrack 0,\,\infty)$, $M\in(1,\,\infty)$, and $r_{1},\,r_{2}\in(0,\,R\rbrack$ with $r_{1}<r_{2}$, provided $V_{\varepsilon}\in L^{\alpha+p}(Q_{r_{2}})$. 
In fact, letting $M\to\infty$ in (\ref{Eq (Section 5): Reversed Hoelder Claim}) and recalling (\ref{Eq (Section 2): psi-tilde estimate 2}), we have 
\begin{align*}
&\iint_{Q_{r_{1}}}V_{\varepsilon}^{\alpha+2}\,{\mathrm d}x{\mathrm d}t \\ 
&\le \iint_{Q_{r_{1}}}\left(V_{\varepsilon}^{\alpha+p}+1\right)\,{\mathrm d}x{\mathrm d}t+\limsup_{M\to\infty}\iint_{Q_{r_{1}}} {\tilde \psi}_{\alpha,\,M}(V_{\varepsilon})V_{\varepsilon}^{2}\,{\mathrm d}x{\mathrm d}t\\ 
&\le \frac{C(n,\,p,\,\lambda,\,\Lambda,\,K,\,\alpha,\,M_{0})}{(r_{2}-r_{1})^{2}}\iint_{Q_{r_{2}}}\left(V_{\varepsilon}^{\alpha+p}+1\right)\,{\mathrm d}x{\mathrm d}t
\end{align*}
by Beppo Levi's monotone convergence theorem.
By this estimate and an iteration argument in finitely many steps, for any $m\in{\mathbb N}$, we can deduce (\ref{Eq (Section 5): Reversed Hoelder}) with $q=p+(2-p)m$. 
Therefore, by H\"{o}lder's inequality, it is easy to verify (\ref{Eq (Section 5): Reversed Hoelder}) for arbitrary $q\in(p,\,\infty)$.

To prove (\ref{Eq (Section 5): Reversed Hoelder Claim}), we choose $\eta$ and $\phi_{\mathrm c}$ satisfying (\ref{Eq (Section 4): Choice of Cut-Off}). 
Integrating by parts, we obtain
\begin{align*}
& \iint_{Q_{r_{1}}}{\tilde \psi}_{\alpha,\,M}(V_{\varepsilon})V_{\varepsilon}^{2}\,{\mathrm d}x{\mathrm d}t=\iint_{Q_{r_{2}}} {\tilde \psi}_{\alpha,\,M}(V_{\varepsilon}) V_{\varepsilon}^{2} \eta^{2}\phi_{\mathrm c}\,{\mathrm d}x{\mathrm d}t  \\ 
&\le \iint_{Q_{r_{2}}}{\tilde \psi}_{\alpha,\,M}(V_{\varepsilon}) \eta^{2}\phi_{\mathrm c}\,{\mathrm d}x{\mathrm d}t +C_{n}\iint_{Q_{r}}{\tilde\psi}_{\alpha,\,M}(V_{\varepsilon}) \lvert u_{\varepsilon}\rvert\lvert\nabla\eta\rvert\eta\phi_{\mathrm c}\,{\mathrm d}x{\mathrm d}t\\ 
&\quad +C_{n}\int_{Q_{r_{2}}}\left({\tilde \psi}_{\alpha,\,M}(V_{\varepsilon})\lvert \nabla^{2}u_{\varepsilon}\rvert +{\tilde \psi}_{\alpha,\,M}^{\prime}(V_{\varepsilon})\lvert\nabla V_{\varepsilon}\rvert V_{\varepsilon}\right)\lvert u_{\varepsilon}\rvert\eta^{2}\phi_{\mathrm c}\,{\mathrm d}x{\mathrm d}t\\ 
&\le \frac{C(n,\,M_{0})}{r_{2}-r_{1}}\iint_{Q_{r_{2}}}{\tilde \psi}_{\alpha,\,M}(V_{\varepsilon})\,{\mathrm d}x{\mathrm d}t \\ 
&\quad +C_{n}M_{0}\left[\iint_{Q_{r_{2}}}V_{\varepsilon}^{p-2}\left(\lvert \nabla^{2} u_{\varepsilon}\rvert^{2}{\tilde \psi}_{\alpha,\,M}(V_{\varepsilon})+\lvert\nabla V_{\varepsilon}\rvert^{2}{\tilde\psi}_{\alpha,\,M}^{\prime}(V_{\varepsilon})V_{\varepsilon} \right)\eta^{2}\phi_{\mathrm{c}}\,{\mathrm d}x{\mathrm d}t \right]^{1/2}\\ 
&\quad \quad \cdot\left[\iint_{Q_{r_{2}}}V_{\varepsilon}^{2-p}\left({\tilde\psi}_{\alpha,\,M}(V_{\varepsilon})+{\tilde\psi}_{\alpha,\,M}^{\prime}(V_{\varepsilon})V_{\varepsilon} \right) \eta^{2}\phi_{\mathrm c}\,{\mathrm d}x{\mathrm d}t \right]^{1/2},
\end{align*}
where we have used (\ref{Eq (Section 5): Bounds of u-epsilon}) and the Cauchy--Schwarz inequality.
By (\ref{Eq (Section 2): psi-tilde estimate 1}), (\ref{Eq (Section 4): Choice of Cut-Off}), (\ref{Eq (Section 5): Energy estimate 1}) and Young's inequality, we have
\begin{align*}
  &\iint_{Q_{r_{1}}}{\tilde \psi}_{\alpha,\,M}(V_{\varepsilon})V_{\varepsilon}^{2}\,{\mathrm d}x{\mathrm d}t\\ 
  &\le \frac{1}{4}\iint_{Q_{r_{2}}}{\tilde \psi}_{\alpha,\,M}(V_{\varepsilon})V_{\varepsilon}^{2}\,{\mathrm d}x{\mathrm d}t+\frac{C}{(r_{2}-r_{1})^{2}}\iint_{Q_{r_{2}}}\left({\tilde \psi}_{\alpha,\,M}(V_{\varepsilon})\left(V_{\varepsilon}^{p}+1\right)+V_{\varepsilon}^{\alpha+2-p}\right)\,{\mathrm d}x{\mathrm d}t\\ 
  &\le \frac{1}{4}\iint_{Q_{r_{2}}}{\tilde \psi}_{\alpha,\,M}(V_{\varepsilon})V_{\varepsilon}^{2}\,{\mathrm d}x{\mathrm d}t+\frac{C}{(r_{2}-r_{1})^{2}}\iint_{Q_{r_{2}}}\left(V_{\varepsilon}^{\alpha+p}+1\right)\,{\mathrm d}x{\mathrm d}t,
\end{align*}
where the constant $C\in(1,\,\infty)$ depends at most on $n$, $p$, $\lambda$, $\Lambda$, $K$, $\alpha$, and $M_{0}$.
The desired estimate (\ref{Eq (Section 5): Reversed Hoelder Claim}) follows from Lemma \ref{Lemma: Absorbing Iteration}.
\end{proof}

From (\ref{Eq (Section 5): Energy estimate 2}), we complete the case $q=\infty$ by Moser's iteration.
\begin{proposition}\label{Prop: Moser's iteration}
  Let $n$ and $p$ satisfy (\ref{Eq (Section 1): Subcritical Range}).
  Fix an exponent $q$ satisfying
  \[q_{\mathrm c}\coloneqq \frac{n(2-p)}{2}<q<\infty,\quad \textrm{and}\quad q\ge 2.\]
  Let $u_{\varepsilon}$ be a weak solution to (\ref{Eq (Section 3): Approximate Equation}) in ${\tilde Q}$.
  Fix $Q_{R}(x_{0},\,t_{0})\Subset {\tilde Q}$ with $R\in(0,\,1)$, and let (\ref{Eq (Section 5): Improved reg}) be in force.
  Then, there holds
  \begin{equation}\label{Eq (Section 5): L-infty L-q}
    \esssup_{Q_{R/2}}V_{\varepsilon}\le C\left(\fiint_{Q_{R}}(1+V_{\varepsilon})^{q} \,{\mathrm d}x{\mathrm d}t \right)^{1/(q-q_{\mathrm c})}
  \end{equation}
  Here the constant $C\in(1,\,\infty)$ depends at most on $n$, $p$, $q$, $\lambda$, $\Lambda$, and $K$.
\end{proposition}
\begin{proof}
 To prove (\ref{Eq (Section 5): L-infty L-q}), we claim that for every $\beta\in \lbrack 2,\,\infty)$, there holds
 \begin{equation}\label{Eq (Section 5): Reversed Hoelder for Moser iteration}
   \iint_{Q_{r_{1}}}W_{\varepsilon}^{\kappa\beta+p-2} \le \left[\frac{C(\beta-1)^{\gamma}}{(r_{2}-r_{1})^{2}} \iint_{Q_{r_{2}}}W_{\varepsilon}^{\beta}\,{\mathrm d}x{\mathrm d}t\right]^{\kappa},
 \end{equation}
where $\kappa\coloneqq 1+2/n$, $\gamma\coloneqq 2(1+1/(n+2))$, and the constant $C\in(1,\,\infty)$ depends at most on $n$, $p$, $\lambda$, $\Lambda$, $K$.
Fix $r_{1},\,r_{2}\in(0,\,R\rbrack$ with $r_{1}<r_{2}$, and choose $\eta$ and $\phi_{\mathrm{c}}$ satisfying (\ref{Eq (Section 4): Choice of Cut-Off}).
By (\ref{Eq (Section 2): psi estimate 0}) with $r=2$, (\ref{Eq (Section 5): Energy estimate 2}), H\"{o}lder's inequality and the continuous embedding $W_{0}^{1,\,2}(B_{r_{2}})\hookrightarrow L^{\frac{2n}{n-2}}(B_{r_{2}})$, we obtain
\begin{align*}
  &\iint_{Q_{r_{1}}}\Psi_{\alpha,\,M}(W_{\varepsilon})^{\frac{2}{n}}\psi_{\alpha,\,M}(W_{\varepsilon})W_{\varepsilon}^{p} \\ 
  &\le C_{n}\left(\esssup_{t_{0}-r_{2}^{2}<\tau<t_{0}}\int_{B_{r_{2}}\times \{\tau\}}\Psi_{\alpha,\,M}(W_{\varepsilon})\eta^{2}\phi_{\mathrm c}\,{\mathrm d}x\right)^{\frac{2}{n}} \iint_{Q_{r_{2}}}\left\lvert\nabla\left(\eta\psi_{\alpha,\,M}(W_{\varepsilon})^{\frac{1}{2}}W_{\varepsilon}^{\frac{p}{2}}\right) \right\rvert^{2}\phi_{\mathrm c} \,{\mathrm d}x{\mathrm d}t\\ 
  &\le C(n,\,p,\,\lambda,\,\Lambda,\,K)\left[\frac{(1+\alpha)^{2}}{(r_{2}-r_{1})^{2}}\iint_{Q_{r_{2}}}\psi_{\alpha,\,M}(W_{\varepsilon})W_{\varepsilon}^{2}\,{\mathrm d}x{\mathrm d}t\right]^{\kappa},
\end{align*}
where we note $W_{\varepsilon}\ge 1$ and therefore $W_{\varepsilon}^{p}\le W_{\varepsilon}^{2}$.
Letting $M\to \infty$ and recalling (\ref{Eq (Section 2): Monotone conv of psi-s}), we conclude (\ref{Eq (Section 5): Reversed Hoelder for Moser iteration}) by Beppo Levi's monotone convergence theorem.

We set the sequences $\{q_{l}\}_{l=0}^{\infty}\subset \lbrack q,\,\infty)$, $\{R_{l}\}_{l=0}^{\infty}\subset (R/2,\,R\rbrack$, $\{Y_{l}\}_{l=0}^{\infty}\subset {\mathbb R}_{\ge 0}$ as
\[q_{l}\coloneqq \mu\kappa^{l}+q_{\mathrm c},\quad R_{l}\coloneqq \frac{1+2^{-l}}{2}R,\quad \text{and}\quad Y_{l}\coloneqq \left(\iint_{Q_{R_{l}}}W_{\varepsilon}^{q_{l}}\,{\mathrm d}x{\mathrm d}t\right)^{1/q_{l}}\]
for each $l\in{\mathbb Z}_{\ge 0}$, where $\mu\coloneqq q-q_{\mathrm c}$.
Then, similarly to Proposition \ref{Prop: L-infty bounds}, we can find the constant ${\tilde \kappa}={\tilde \kappa}(\kappa,\,q,\,q_{\mathrm c})\in(\kappa,\,\infty)$ such that (\ref{Eq (Section 4): Iteration on exponents}) holds for every $l\in{\mathbb Z}_{\ge 0}$.
Hence, (\ref{Eq (Section 5): Reversed Hoelder for Moser iteration}) with $\beta\coloneqq q_{l}\ge q_{0}=q\ge 2$ yields
\[Y_{l+1}^{q_{l+1}}\le \left(AB^{l}Y_{l}^{q_{l}}\right)^{\kappa},\quad q_{l}\ge \mu\left(\kappa^{l}-1\right)\quad \text{for all }l\in{\mathbb Z}_{\ge 0},\]
where we set $A\coloneqq CR^{-2}\in(1,\,\infty)$ for some constant $C\in(1,\,\infty)$, and $B\coloneqq 4{\tilde \kappa}^{\gamma}\in(1,\,\infty)$.
By applying Lemma \ref{Lemma: Moser iteration lemma} and recalling (\ref{Eq (Section 5): Compatibility}), we have 
\begin{align*}
  &\esssup_{Q_{R/2}}\,V_{\varepsilon}\le C_{n}\esssup_{Q_{R/2}}\,W_{\varepsilon}\le C_{n}\limsup_{l\to\infty}\,Y_{l}\\ 
  &\le C(n,\,p,\,q,\,\lambda,\,\Lambda,\,K)\left(R^{-2\kappa^{\prime}}\iint_{Q_{R}}W_{\varepsilon}^{q}\,{\mathrm d}x{\mathrm d}t \right)^{1/\mu}\\ 
  &\le C(n,\,p,\,q,\,\lambda,\,\Lambda,\,K)\left(\fiint_{Q_{R}}\left(1+V_{\varepsilon}\right)^{q}\,{\mathrm d}x{\mathrm d}t \right)^{1/(q-q_{\mathrm c})},
\end{align*}
which completes the proof.
\end{proof}

Section \ref{Sect: L-p to L-infty} is completed by showing Theorem \ref{Thm: Grad Bound}.
\begin{theorem}\label{Thm: Grad Bound}
  Let $n$ and $p$ satisfy (\ref{Eq (Section 1): Subcritical Range}).
  Assume that $u_{\varepsilon}$ is a weak solution to (\ref{Eq (Section 3): Approximate Equation}) in ${\tilde Q}$. 
  Fix $Q_{R}(x_{0},\,t_{0})\Subset {\tilde Q}$ with $R\in(0,\,1)$, and let (\ref{Eq (Section 5): Bounds of u-epsilon}) and (\ref{Eq (Section 5): Improved reg}), and $\lVert \nabla u_{\varepsilon}\rVert_{L^{p}(Q_{R})}\le M_{1}$ be in force.
  Here the constant $M_{1}\in(1,\,\infty)$ is independent of $\varepsilon\in(0,\,1)$.
  Then, for each $r\in(0,\,R)$, there exists a constant $C\in(1,\,\infty)$, depending at most on $n$, $p$, $\lambda$, $\Lambda$, $K$, $M_{0}$, $M_{1}$, $R$, and $r$, such that 
  \[\esssup_{Q_{r}}V_{\varepsilon}\le C.\]
\end{theorem}
\begin{proof}
  Choose and fix $q\in(q_{\mathrm c},\,\infty)\cap \lbrack 2,\,\infty)$. 
  By Proposition \ref{Prop: Higher integrability}, there exists a constant ${\tilde C}\in(1,\,\infty)$, depending at most on $n$, $p$, $q$, $\lambda$, $\Lambda$, $K$, $M_{0}$, $M_{1}$, $R$, and $r$, such that we have
  \[\lVert V_{\varepsilon} \rVert_{L^{q}(Q_{\tilde R})}\le {\tilde C},\quad \text{where}\quad {\tilde R}\coloneqq \frac{R+r}{2}.\]
  Combining this bound with Proposition \ref{Prop: Moser's iteration}, we conclude Theorem \ref{Thm: Grad Bound}.
\end{proof}

\section{The proof of main theorem}\label{Sect:MainTheorem}
\subsection{A priori continuity estimates of truncated gradients}\label{Subsect: A priori Hoelder}
We infer a basic result of a priori H\"{o}lder estimates.
\begin{theorem}\label{Thm: A priori Hoelder}
Fix $\delta\in(0,\,1)$, and let $\varepsilon\in(0,\,\delta/8)$.
Assume that $u_{\varepsilon}$ is a weak solution to (\ref{Eq (Section 3): Approximate Equation}) in ${\tilde Q}$, and let (\ref{Eq (Section 5): Improved reg}) be in force for a fixed $Q_{R}=Q_{R}(x_{0},\,t_{0})\Subset{\tilde Q}$.
Also, let the positive number $\mu_{0}$ satisfy 
\begin{equation}\label{Eq (Section 6): Lip bounds}
  \esssup_{Q_{r}} V_{\varepsilon}\le \delta+\mu_{0},
\end{equation}
where $Q_{r}=Q_{r}(x_{\ast},\,t_{\ast})\Subset Q_{R}$ with $r\in(0,\,1)$.
Then, there hold
\[\left\lvert {\mathcal G}_{2\delta,\,\varepsilon}(\nabla u_{\varepsilon}) \right\rvert\le \mu_{0}\quad \textrm{in}\quad Q_{r_{0}}(x_{\ast},\,t_{\ast}),\] 
and
\[\left\lvert {\mathcal G}_{2\delta,\,\varepsilon}(\nabla u_{\varepsilon}(X_{1}))-{\mathcal G}_{2\delta,\,\varepsilon}(\nabla u_{\varepsilon}(X_{2})) \right\rvert\le C\left(\frac{d_{\mathrm{p}}(X_{1},X_{2})}{r_{0}}\right)^{\alpha}\mu_{0}\]
for all $X_{1}=(x_{1},\,t_{1})$, $X_{2}=(x_{2},\,t_{2})\in Q_{r_{0}/2}(x_{\ast},\,t_{\ast})$.
Here the radius $r_{0}\in(0,\,r/4)$, the exponent $\alpha\in(0,\,1)$, and the constant $C\in(1,\,\infty)$ depend at most on $n$, $p$, $\lambda$, $\Lambda$, $K$, $\mu_{0}$, and $\delta$.
\end{theorem}
As long as (\ref{Eq (Section 6): Lip bounds}) is guaranteed, Theorem \ref{Thm: A priori Hoelder} is shown as a special case of \cite[Theorem 2.8]{T-parabolic}, where an external force term is also treated.
We briefly note how the proof of \cite[Theorem 2.8]{T-parabolic} is carried out, in the special case of no external force term.
To explain the brief strategy, we go back to the weak formulation (\ref{Eq (Section 5): Weak Form}), which tells us two pieces of basic information.

The first is that the convex composite function $\Psi(V_{\varepsilon})$ is a weak subsolution to a parabolic equation, which is easy to notice by neglecting the integrals on the second line of (\ref{Eq (Section 5): Weak Form}).
When we choose $\psi(\sigma)\coloneqq 2(1-\delta/\sigma)_{+}$, then the composite function $\Psi(V_{\varepsilon})$ becomes $\lvert{\mathcal G}_{\delta,\,\varepsilon}(\nabla u_{\varepsilon})\rvert^{2}\eqqcolon U_{\delta,\,\varepsilon}$, where
\[{\mathcal G}_{\delta,\,\varepsilon}(\nabla u_{\varepsilon})\coloneqq \left(V_{\varepsilon}-\delta\right)_{+}\frac{\nabla u_{\varepsilon}}{\lvert\nabla u_{\varepsilon}\rvert}.\]
Noting that the support of $U_{\delta,\,\varepsilon}$ is contained in $\{V_{\varepsilon}\ge \delta\}$, where the equation becomes uniformly parabolic in the classical sense \cite[Lemma 3.2]{T-parabolic}, we can deduce a De Giorgi-type oscillation lemma \cite[Proposition 2.9]{T-parabolic}, provided that a sublevel set of $V_{\delta,\,\varepsilon}$ is suitably non-degenerate.

The second is that we can obtain local $L^{2}$-estimates related to the second order derivative $\nabla^{2}u_{\varepsilon}$.
In particular, by choosing $\psi(\sigma)$ behaving like $\sigma^{p}$, we can obtain local $L^{2}$-bounds concerning $\nabla (V_{\varepsilon}^{p-1}\nabla u_{\varepsilon})$ \cite[Lemma 3.3]{T-parabolic}.
By this local energy estimate, and the comparison with a heat flow, we can deduce a Campanato-type decay estimate \cite[Proposition 2.10]{T-parabolic}, provided a superlevelset of $V_{\varepsilon}$ is sufficiently large.
There, we also rely on the continuity of the mapping $z\mapsto\nabla^{2}E^{\varepsilon}  (z)$ over an annulus region $\{\delta\le \lvert z\rvert \le M\}$ for some fixed constant $M$.
Hence, (\ref{Eq (Section 1): modulus of continuity of Hess Ep}) and (\ref{Eq (Section 1): (Assumption) modulus of continuity of Hess E1}), which are never used in Sections \ref{Sect:Approximation}--\ref{Sect: L-p to L-infty}, are required in the proof of \cite[Proposition 2.10]{T-parabolic}.

In any possible cases, it is proved that the limit 
\[\Gamma_{2\delta,\,\varepsilon}(x_{0},\,t_{0})\coloneqq \lim_{r\to 0}\fiint_{Q_{r}(x_{0},\,t_{0})}{\mathcal G}_{2\delta,\,\varepsilon}(\nabla u_{\varepsilon})\,{\mathrm d}x{\mathrm d}t\in{\mathbb R}^{n}\]
exists for every $(x_{0},\,t_{0})\in Q_{r_{0}}(x_{\ast},\,t_{\ast})$. 
Here the radius $r_{0}$ and the exponent $\alpha$ are determined by \cite[Propositions 2.9--2.10]{T-parabolic}, depending on $n$, $p$, $\lambda$, $\Lambda$, $K$, $\omega_{1}$, $\omega_{p}$, $\mu_{0}$, and $\delta$.
Moreover, this limit satisfies
\[\fiint_{Q_{\rho}(x_{0},\,t_{0})}\left\lvert{\mathcal G}_{2\delta,\,\varepsilon}(\nabla u_{\varepsilon})-\Gamma_{2\delta,\,\varepsilon}(x_{0},\,t_{0})\right\rvert^{2}\,{\mathrm d}x{\mathrm d}t\le C\left(\frac{\rho}{r_{0}}\right)^{2\alpha}\]
for any radius $\rho\in (0,\,r_{0}\rbrack$.
Theorem \ref{Thm: A priori Hoelder} follows from this growth estimate.

To carry out classical arguments, including De Giorgi's truncation and a comparison argument in \cite[Propositions 2.9--2.10]{T-parabolic}, we often use the assumption \(\delta<\mu\).
Here $\mu$ is a positive parameter that satisfies 
\[\esssup_{Q}\,\lvert {\mathcal G}_{\delta,\,\varepsilon}(\nabla u_{\varepsilon})\rvert \le \mu\] 
for a fixed cylinder $Q\Subset Q_{r}(x_{\ast},\,t_{\ast})$.
The condition $\delta<\mu$ is not restrictive, since otherwise ${\mathcal G}_{2\delta,\,\varepsilon}(\nabla u_{\varepsilon})\equiv 0$ in $Q$, and hence there is nothing to show in the proof of Theorem \ref{Thm: A priori Hoelder}.
This truncation trick is substantially different from the intrinsic scaling arguments.
It should be emphasized that the intrinsic scaling argument relies on some uniform parabolicity of the $p$-Laplace operator, which can be measured by an analogous ratio defined as in the left-hand side of (\ref{Eq (Section 1): PR}).
To the contrary, for (\ref{Eq (Section 3): Approximate Equation}), its uniform parabolicity basically depends on the value of $V_{\varepsilon}$.
For this reason, in the proof of \cite[Theorem 2.8 and Propositions 2.9--2.10]{T-parabolic}, we avoid rescaling a weak solution $u_{\varepsilon}$.
Instead, we make the truncation with respect to the modulus $V_{\varepsilon}$, which enables us to treat (\ref{Eq (Section 3): Approximate Equation}) as it is some sort of uniform parabolic equation, depending on the truncation parameter $\delta\in(0,\,1)$.
This is verified by the assumption $\delta<\mu$, which plays an important role in deducing non-trivial energy estimates \cite[Lemmata 3.2--3.3]{T-parabolic}.

\subsection{Proof of main theorem}\label{Subsect: Last Subsection}
We conclude the paper by giving the proof of Theorem \ref{Thm}.
\begin{proof}
We fix parabolic cylinders $Q_{(0)}\Subset Q_{(1)}\Subset Q_{(2)} \Subset Q_{(3)} \Subset Q_{(4)}\Subset \Omega_{T}$ arbitrarily, and we claim the local H\"{o}lder continuity of ${\mathcal G}_{2\delta}(\nabla u)$ in $Q_{(0)}$.
We abbreviate
\[d\coloneqq \min\left\{\,\mathop{\mathrm{dist}_{\mathrm p}}\left(\partial_{\mathrm{p}}Q_{(k+1)},\,\partial_{\mathrm{p}}Q_{(k)}\right)\mathrel{}\middle|\mathrel{} k=0,\,1,\,2,\,3\,\right\}>0.\]
By Proposition \ref{Prop: L-infty bounds}, we have $\lVert u\rVert_{L^{\infty}(Q_{(3)})}\le M_{0}$, where the constant $M_{0}\in(1,\,\infty)$ depends at most on $n$, $p$, $s$, $\lVert u\rVert_{L^{s}(Q_{(4)})}$, and $d$.
For each fixed $\varepsilon\in(0,\,\delta/8)$, let $u_{\varepsilon}$ be the weak solution of (\ref{Eq (Section 3): Dirichlet Prob approx}) with $\Omega_{T}$ replaced by $Q_{(3)}$.
By Proposition \ref{Prop: Convergence}, there exists a constant $M_{1}$ such that we have
\[\lVert \nabla u_{\varepsilon}\rVert_{L^{p}(Q_{(2)})}\le M_{1}.\]
Also, we are allowed to choose a sequence $\{\varepsilon_{j}\}_{j}$, satisfying $\varepsilon_{j}\to 0$ as $j\to\infty$, such that \(\nabla u_{\varepsilon}\to \nabla u\) a.e. in $Q_{(2)}$.
In particular, it follows that
\begin{equation}\label{Eq (Section 2): Conv on G-2delta-epsilon}
{\mathcal G}_{2\delta,\,\varepsilon_{j}}(\nabla u_{\varepsilon_{j}})\to {\mathcal G}_{2\delta}(\nabla u)
\end{equation}
a.e.~in $Q_{(2)}$.
We should note $\lVert u\rVert_{L^{\infty}(Q_{(2)})}\le M_{0}$ by Corollary \ref{Lemma: WMP}.
Therefore, we can apply Theorem \ref{Thm: Grad Bound} to find the constant $\mu_{0}\in (1,\,\infty)$, which depends at most on $n$, $p$, $\lambda$, $\Lambda$, $K$, $M_{0}$, $M_{1}$, and $d$, such that (\ref{Eq (Section 6): Lip bounds}) holds for any $Q_{r}(x_{\ast},\,t_{\ast})\Subset Q_{(1)}$ with $r\in(0,\,1)$.
By Theorem \ref{Thm: A priori Hoelder} and a standard covering argument, we can apply the Arzel\'{a}--Ascoli theorem to ${\mathcal G}_{2\delta,\,\varepsilon}(\nabla u_{\varepsilon})\in C^{0}(Q_{(0)};\,{\mathbb R}^{n})$.
As a consequence, we may let the convergence (\ref{Eq (Section 2): Conv on G-2delta-epsilon}) hold uniformly in $Q_{(0)}$ by taking a subsequence if necessary.
Hence, for each fixed $\delta\in(0,\,1)$, ${\mathcal G}_{2\delta}(\nabla u)$ is H\"{o}lder continuous in $Q_{(0)}\Subset \Omega_{T}$, which completes the proof of ${\mathcal G}_{2\delta}(\nabla u)\in C^{0}(\Omega_{T};\,{\mathbb R}^{n})$.

By the definition of ${\mathcal G}_{\delta}$, it is easy to check that $\{{\mathcal G}_{2\delta}(\nabla u)\}_{\delta\in(0,\,1)}\subset C^{0}(\Omega_{T};\,{\mathbb R}^{n})$ is a Cauchy net, and therefore this has a uniform convergence limit $v_{0}\in C^{0}(\Omega_{T};\,{\mathbb R}^{n})$ as $\delta\to 0$. 
Combining with ${\mathcal G}_{2\delta}(\nabla u)\to\nabla u$ a.e.~in $\Omega_{T}$, we conclude $\nabla u=v_{0}\in C^{0}(\Omega_{T};\,{\mathbb R}^{n})$.
\end{proof}

\subsection*{Acknowledgment}
During the preparation of the paper, the author was partly supported by the Japan Society for the Promotion of Science through JSPS KAKENHI Grant Number 22KJ0861.

\addcontentsline{toc}{section}{References}
\bibliographystyle{plain}
\end{document}